\documentclass[11pt]{article}
\usepackage{amssymb}
\usepackage{amsmath}
\usepackage{graphicx}
\usepackage{setspace}

\begin{document}

\newtheorem{theorem}{Theorem}[section]
\newtheorem{tha}{Theorem}
\newtheorem{conjecture}[theorem]{Conjecture}
\newtheorem{corollary}[theorem]{Corollary}
\newtheorem{lemma}[theorem]{Lemma}
\newtheorem{claim}[theorem]{Claim}
\newtheorem{example}[theorem]{Example}
\newtheorem{proposition}[theorem]{Proposition}
\newtheorem{construction}[theorem]{Construction}
\newtheorem{definition}[theorem]{Definition}
\newtheorem{question}[theorem]{Question}
\newtheorem{problem}[theorem]{Problem}
\newtheorem{remark}[theorem]{Remark}
\newtheorem{observation}[theorem]{Observation}

\newcommand{\ex}{{\mathrm{ex}}}

\newcommand{\EX}{{\mathrm{EX}}}

\newcommand{\Ker}{{\mathrm{Ker}}}

\newcommand{\AR}{{\mathrm{AR}}}

\def\endproofbox{ \qed}
\newenvironment{proof}%
{%
\noindent{\it Proof.}
}%
{%
 \quad\hfill\endproofbox 
}
       \def\qed{\ifhmode\unskip\nobreak\hfill$\Box$\medskip\fi\ifmmode\eqno{\Box}\fi}
\def\ce#1{\lceil #1 \rceil}
\def\fl#1{\lfloor #1 \rfloor}
\def\lr{\longrightarrow}
\def\e{\epsilon}
\def\ve{\varepsilon}
\def\cA{{\cal A}}
\def\cB{{\cal B}}
\def\cC{{\cal C}}
 \def\cD{{\cal D}}
\def\cF{{\cal F}}
\def\cG{{\cal G}}
\def\cH{{\cal H}}
\def\cK{{\cal K}}
\def\cI{{\cal I}}
\def\cJ{{\cal J}}
\def\cL{{\cal L}}
\def\cM{{\cal M}}
\def\cP{{\cal P}}
\def\cQ{{\cal Q}}
\def\cR{{\cal R}}
\def\cS{{\cal S}}
\def\cT{{\cal T}}
\def\cW{{\cal W}}
\def\cZ{{\cal Z}}

\def\bkl{{\cal B}^{(k)}_\ell}
\def\cmkt{{\cal M}^{(k)}_{t+1}}
\def\cpkl{{\cal P}^{(k)}_\ell}
\def\cckl{{\cal C}^{(k)}_\ell}

\def\bC{\mathbb{C}}
\def\pkl{\mathbb{P}^{(k)}_\ell}
\def\ckl{\mathbb{C}^{(k)}_\ell}
\def\cthreel{\mathbb{C}^{(3)}_\ell}

\def\imp{\Longrightarrow}
\def\1e{\frac{1}{\e}\log \frac{1}{\e}}
\def\ne{n^{\e}}
\def\rad{ {\rm \, rad}}
\def\equ{\Longleftrightarrow}
\def\pkl{\mathbb{P}^{(k)}_\ell}

\def\c2ell{\mathbb{C}^{(2)}_\ell}
\def\codd{\mathbb{C}^{(k)}_{2t+1}}
\def\ceven{\mathbb{C}^{(k)}_{2t+2}}
\def\podd{\mathbb{P}^{(k)}_{2t+1}}
\def\peven{\mathbb{P}^{(k)}_{2t+2}}
\def\TT{{\mathbb T}}
\def\bbT{{\mathbb T}}
\def\cE{{\mathcal E}}

\def\wt{\widetilde}
\def\wh{\widehat}
\def\wS{\widetilde{S}}
\def\whf{\widehat{\cF}}
\def\tF{\widetilde{\cF}}
\voffset=-0.5in
	
\setstretch{1.1}
\pagestyle{myheadings}
\markright{{\small \sc F\"uredi and Jiang:}
  {\it\small Tur\'an numbers of hypergraph trees}}

\title{\huge\bf Tur\'an numbers of hypergraph trees}

\author{
Zolt\'an F\"uredi\thanks{Department of Mathematics, University of Illinois, Urbana, IL 61801, USA. E-mail:
z-furedi@illinois.edu
\newline
Research supported in part by the Hungarian National Science Foundation OTKA 104343,
 by the Simons Foundation Collaboration Grant 317487,
and by the European Research Council Advanced Investigators Grant 267195.
}\quad\quad
Tao Jiang\thanks{Department of Mathematics, Miami University, Oxford,
OH 45056, USA. E-mail: jiangt@miamioh.edu.
Research supported in part by the Simons Foundation Collaboration Grant \#282906
and by National Science Foundation Grant DMS-1400249.\newline 
Both authors are grateful to the hospitality of Mittag-Leffler Institute during their respective visits in Spring 2014.
  %
} 
} 

\date{May 12, 2015}   

\maketitle
\begin{abstract}
An $r$-graph is an {\it $r$-uniform hypergraph tree} (or {\it $r$-tree}) if its edges can be ordered as $E_1,\ldots, E_m$ such that $\forall i>1 \, \exists \alpha(i)<i$ such that  $E_i\cap (\bigcup_{j=1}^{i-1} E_j)\subseteq E_{\alpha(i)}$.
The {\it Tur\'an number} $\ex(n,\cH)$ of an $r$-graph $\cH$ is the largest size of an $n$-vertex $r$-graph that does not contain $\cH$.  A {\it cross-cut} of $\cH$ is a set of vertices in $\cH$ that contains exactly one vertex of each edge of $\cH$.
The {\it cross-cut number} $\sigma(\cH)$ of $\cH$ is the minimum size of a cross-cut of $\cH$. We show that for a large family of $r$-graphs (largest within a certain scope) that are embeddable in $r$-trees, $\ex(n,\cH)=(\sigma-1)\binom{n}{r-1}+o(n^{r-1})$ holds, and we establish structural stability of near extremal graphs.
From stability, we establish exact results for some subfamilies.

\end{abstract}

\section{Introduction, reducible and embeddable hypertrees}

In this paper, $r$-graphs refer to $r$-uniform hypergraphs. We will use $\cF\subseteq \binom{V}{r}$ to indicate that $\cF$ is an $r$-graph on the set $V$, usually we take $V=[n]:=\{1,\ldots, n\}$.
Given an $r$-graph $\cH$ and a positive integer $n$, the Tur\'an number $\ex(n,\cH)$
is the largest size of an $r$-graph on $n$ vertices not containing $\cH$ as a subgraph.
While the study of hypergraph Tur\'an numbers is a notoriously difficult area of extremal combinatorics 
there have been active developments in recent years concerning ``tree-like'' hypergraphs.

The classic Erd\H{o}s-Ko-Rado theorem~\cite{EKR} determines the maximum size of
an $r$-graph not containing two disjoint edges (i.e., a matching of size $2$). See
~\cite{EG65, frankl-matching, FLM, FRR, HLS, LM} for recent work. Note that a matching is an $r$-tree.
There are many theorems and conjectures inspired by
and/or generalize the Erd\H{o}s-Ko-Rado theorem, including results on set systems
not containing 
 {\em $d$-clusters} and their generalizations~\cite{FF-cluster, Furedi-Ozkahya, Mubayi-cluster, KM,  MR} and on set systems not containing 
 {\em $d$-simplex} or {\em strong $d$-simplex}~\cite{chvatal, FF-exact, MR, JPY}.

An $r$-tree $\cH$ is a {\it tight $r$-tree} if its edges can be
ordered as $E_1,\ldots, E_m$ such that $\forall i>1 \, \exists \alpha(i)<i$ one has $E_i\cap (\bigcup_{j<i} E_j)\subseteq E_{\alpha(i)}$
 and $|E_i\cap E_{\alpha(i)}|=r-1$ (hence $|E_i\setminus (\bigcup_{j<i}E_j)|=1$).
\begin{conjecture} {\bf(Erd\H{o}s-S\'os for graphs and Kalai 1984 for $r\geq 3$)} \label{kalai}
Let $r\geq 2$ and $\cH$ a tight $r$-tree on $v$ vertices. Then $$\ex(n,\cH)=\frac{v-r}{r}\binom{n}{r-1}+o(n^{r-1}).$$
\end{conjecture}

To this date the conjecture was verified only for star-shaped tight $r$-trees by Frankl and F\"uredi~\cite{FF-exact}. 
While Kalai's conjecture is wide open,
asymptotically tight or even exact results have been obtained
concerning some families of {\it reducible} hypergraph trees and hypergraphs  that are embeddable in them~\cite{BK, Furedi-trees, FJ-cycles, FJS, KMV-path-cycle, KMV-trees}. We say that a hypergraph $\cH$ is {\it $k$-reducible} if each edge of $\cH$ contains at least $k$ degree $1$ vertices.

A hypergraph $(V(\cH),\cH)$ is said to be {\it embeddable} in a hypergraph $(V(\cG),\cG)$ if it is a subgraph of it, more precisely
 there is a mapping $f:V(\cH)\to V(\cG)$ such that $f(E)\in \cG$ for all $E\in \cH$.

\begin{example}[An $r$-graph embeddable  in an $r$-tree may not be an $r$-tree itself]\label{example:1}${}$\\
The $3$-uniform {\it linear cycle} of length $m$ is a hypergraph $\cH:= \{a_ia_{i+1}b_i: 0\leq i\leq m-1\}$ (subscripts taken {\rm mod} $m$).
It is not a $3$-tree though it is embeddable in a $3$-tree
  $\{a_0a_ia_{i+1}: 1\leq i\leq m-2\} \cup \cH$.
  \end{example}

It is easy to see that $\ex(n,\cH)\leq (p-r)\binom{n}{r-1}$ for any $r$-tree on $p$ vertices (see Proposition~\ref{tree-bound}) and so the same holds for any hypergraph $\cH$ embeddable in an $r$-tree.
On the other hand, $\binom{n-1}{r-1}\leq \ex(n, \cH)$ for any $r$-graph with $\bigcap_{F\in \cH} F =\emptyset$.

In this paper, we are interested to determine Tur\'an numbers asymptotically in this Erd\H{o}s-Ko-Rado zone, i.e.,
 when $\ex(n,\cH)=\Theta(n^{r-1})$.
We substantially extend and generalize most of the recent results by asymptotically determining $\ex(n,\cH)$ for all graphs $\cH$ that are embeddable in a $2$-reducible $r$-tree, when $r\geq 4$.
We also obtain structural stability of near extremal graphs and use this to determine the exact value of $\ex(n,\cH)$ for certain graphs.
We describe these results in details in Section~\ref{s:main}.

Let us note that simple $r$-graphs that are embeddable in $r$-trees are $1$-degenerate (i.e. every subgraph has a vertex of degree at most $1$). Answering a question of Erd\H{o}s, F\"uredi ~\cite{furedi-1984} showed that $\ex(n,\cG)=\Theta(n^2)$
for the hypergraph $\cG:=\{123, 124, 356, 456 \}$. Note  that $\cG$ is not $1$-degenerate. Furthermore, $\cG$ is $2$-regular. See \cite{MV}, \cite{MV2}, \cite{MV1}, and~\cite{JL} for some related work. In general, there are many $r$-graphs $\cH$ with  $\ex(n,\cH)=\Theta(n^{r-1})$ that are not embeddable in $r$-trees. 
There is a  lot more to be done concerning Tur\'an type problems for such $r$-graphs $\cH$.

\section{Definitions and Notation} \label{s:definitions}

A {\it hypergraph} $\cH=(V,\cE)$ is a finite set $V$ (called {\it vertices}) and a collection $\cE$ of subsets of $V$ (the {\it edge-set} of $\cH$).
We do not allow multiple edges (we call these {\it simple} hypergraphs) unless otherwise stated.
Many times we identify a simple hypergraph by its edge-set, and write about hypergraph $\cE$. 

A hypergraph $\cG$ (with multiple copies of the same edge allowed) is a {\it hypergraph tree} if its edges can be ordered as
$E_1, E_2,\ldots, E_m$ such that $\forall i>1$, there exists $\alpha(i)<i$ such that
$E_i\cap (\bigcup_{j<i} E_j)\subseteq E_{\alpha(i)}$. Even though there may exist
more than one edge that could serve as $E_{\alpha(i)}$, we will always implicitly
fix a choice in advance so that $\alpha$ is a function. We call $E_{\alpha(i)}$
the {\it parent} of $E_i$. We call the ordering $E_1,\ldots, E_m$ a {\it tree-defining ordering} of $\cG$.
The relation $\alpha(i)<i$ defines a partial order on $[m]$.
It is not hard to verify that any linear extension of this order (i.e., a permutation $\pi: [m]\to [m]$ with $\pi(\alpha(i))< \pi (i)$ for all $i\geq 2$)
 $E_{\pi(1)},\ldots, E_{\pi(m)}$ is also a tree-defining ordering of $\cG$.

Suppose that $\cG$ is  hypergraph tree defined by the sequence $E_1,\ldots, E_m$.
Let $\cG'$ be the corresponding simple hypergraph. Then it is a hypergraph tree as well,
 as it can be seen from the list $E'_1,E'_2,\ldots, E'_p$  obtained by keeping only one copy, namely the the first appearance, of each edge.
Due to this, we do not always
explicitly distinguish between a hypergraph tree in which duplicated edges are
allowed and one in which there is no duplicated edge.
     An $r$-uniform hypergraph tree is also called an {\it $r$-tree}.
An $r$-tree is {\it tight} if $\forall i>1, |E_i\cap E_{\alpha(i)}|=r-1$ or equivalently
$|E_i\setminus (\bigcup_{j<i} E_j)|=1$.

If a hypergraph $\cH$ is a subgraph of another hypergraph $\cG$ then we say that
$\cH$ is {\it embedded/embeddable} in $\cG$. 
As we have seen in Example~\ref{example:1} a hypergraph embeddable in a hypergraph tree may not be a hypergraph tree itself.

An $r$-graph $\cH$ is $r$-partite if its vertex set can be partitioned into $r$ sets
$X_1,\ldots, X_r$ such that each edge of $\cH$ contains exactly one vertex from each $X_i$.
We call such a partition {\it compatible with $\cH$} and the parts $X_i$'s are the {\it color classes} of the $r$-partition.
The following can be easily verified using induction.

\begin{proposition}{\bf ($r$-trees are $r$-partite)}
Every $r$-tree is $r$-partite. Every tight $r$-tree has a unique compatible $r$-partition up to the permutation of color classes.
\qed
\end{proposition}

Given a hypergraph $\cH$, a set $S$ of vertices is a {\it vertex cover} of $\cH$
if $S$ contains at least one vertex of each edge of $\cH$. A vertex cover $S$ of $\cH$ is
called a {\it cross-cut} of $\cH$ if it contains exactly one vertex of each edge of $\cH$.
Every hypergraph has a vertex cover (if $\emptyset \notin \cH$). 
Not every hypergraph has a cross-cut.
However, every $r$-partite $r$-graph $\cH$ has at least one cross-cut. Namely, every color class of an $r$-partition of $\cH$ is a cross-cut of $\cH$.  We let $\tau(\cH)$ denote
the minimum size of a vertex cover of $\cH$ and call it the {\it vertex cover number}
of $\cH$. If $\cH$ has cross-cuts, then we let $\sigma(\cH)$ denote the minimum
size of a cross-cut of $\cH$ and call it the {\it cross-cut number} of $\cH$.

Let $\cF$ be a hypergraph on $V=V(\cF)$. 
We define the {\it $p$-shadow} of $\cF$ to be 
$$\partial_p(\cF):=\{D: |D|=p,\, \exists F\in \cF, D\subseteq F\}.$$  
The Lov\'asz'~\cite{L79} version of the Kruskal-Katona theorem states that if
 $\cF$ is an $r$-graph of size $|\cF|=\binom{x}{r}$, where $x\geq r-1$ is a real number,
 then for all $p$ with $1\leq p\leq r-1$ one has
\begin{equation}\label{eq:KK1.1}
  |\partial_p(\cF)|\geq \binom{x}{p}.
  \end{equation}

Given $D\subseteq V(\cF)$, the {\it degree } $\deg_\cF(D)$ of $D$ in $\cF$  is defined as
$$\deg_\cF(D):=|\{F: F\in \cF, D\subseteq F\}|.$$

Given an $r$-graph $\cF$ and integer $i$ with $1\leq i\leq r-1$ let
$$\delta_i(\cF):=\min\{\deg_\cF(D): D\in \partial_i(\cF)\}.$$

A family of sets $F_1,\ldots, F_s$ is said to form an {\it $s$-star}, a {\it sunflower},  or {\it $\Delta$-system}
of size $s$ with {\it kernel} $D$ if $F_i\cap F_j=D$ for all $1\leq i<j\leq s$ and
$\forall i\in [s]$, $F_i\setminus D\neq \emptyset$.
The sets $F_1,\ldots, F_s$ are called the {\it petals} of this $s$-star.
Note that $D=\emptyset$ is allowed.
Let $\cL^r_p$ denote the $r$-uniform sunflower with a single vertex in the kernel and $p$ petals.
We call $\cL^r_p$ an $r$-uniform {\it linear star} with $p$ edges. 

The {\it kernel degree} $\deg^*_\cF(D)$ of $D$ in $\cF$ is defined  as
$$\deg^*_\cF(D):=\max \{s: \cF \mbox{ contains an $s$-star with kernel } D\}.$$
Given a positive integer $s$,  we define the {\it kernel graph} $\Ker_s(\cF)$ of $\cF$
 with threshold $s$ to be
$$\Ker_s(\cF):=\{D\subseteq V(\cF): \deg^*_\cF(D)\geq s\}.$$
For each $1\leq p\leq r-1$, the {\it $p$-kernel graph} $\Ker^{(p)}_s(\cF)$ of $\cF$ with threshold $s$ is defined to be
$$\Ker^{(p)}_s(\cF):=\{D\subseteq V(\cF): |D|=p, \deg^*_\cF(D)\geq s\}.$$
The following fact  follows easily from the definition of $\deg^*_\cF(D)$ and will be used frequently.
\begin{equation}\label{eq:extension}
 \text{Given a set $Y$. If $\deg^*_\cF(D) > |Y|$ then } \exists F\in \cF
  \text{ such that }D\subseteq F \text{ and }(F\setminus D)\cap Y=\emptyset.
  \end{equation}

 Given a hypergraph $\cF$ and
a vertex $x\in V(\cF)$, let $\cL_\cF(x):=\{F\setminus \{x\}: x\in F\in \cF\}$. We call
$\cL_\cF(x)$ the {\it link graph} of $x$ in $\cF$. Given a set $A\subseteq V(\cF)$,
let $\cL_{\cF}(A):=\bigcap_{x\in A} \cL_\cF(x)$. In other words,
\begin{equation*} 
\cL_\cF(A)=\{D\subseteq V(\cF)\setminus A: \forall a\in A, D\cup a\in \cF\}.
\end{equation*}
We call $\cL_\cF(A)$ the {\it common link graph} of $A$ in $\cF$.

 Given two hypergraphs $\cA$ and $\cB$, the {\it product} of $A$ and $B$ is defined as
\begin{equation*}
\cA\times \cB:=\{A\cup B: A\in \cA, B\in \cB\}.
\end{equation*}

If $\cG$ is an $r$-graph 
 with $V=V(\cG)$, then  its {\it complement} $\overline{\cG}$ is the
$r$-graph on $V$ 
 with edge set $\binom{V}{r}\setminus \cG$.

Given an $r$-graph $\cG$, and set $S\subseteq V(\cG)$, the {\it trace of $G$ on $S$},
denoted by $G|_S$ is the hypergraph with edge set $\{E\cap S: E\in \cG\}$ (after eliminating resulting duplicated edges).
Let $\cG-S:=\cG|_{V(\cG)-S}$.

\section{Lower bounds, $\sigma$-tight families, and $\tau$-perfect families} \label{sigma-tau}
In this section, we present two lower bound constructions on $\ex(n,\cH)$;
one based on the vertex cover number of $\cH$ and the other based on the cross-cut number of $\cH$.

Let $n,r,t$ be positive integers. Define
\begin{equation*} 
\cS^r_{n,t}:=\{F: F\in \binom{[n]}{r}, F\cap [t]\neq \emptyset\}
\end{equation*}
and
\begin{equation*} 
\cC^r_{n,t}:=\{F: F\in \binom{[n]}{r}, |F\cap [t]|=1\}=\binom{[t]}{1}\times \binom{[n]\setminus [t]}{r-1}.
\end{equation*}

So, $S^r_{n,t}$ consists of all the $r$-sets in $[n]$ intersecting a given $t$-set and
$\cC^r_{n,t}$ consists of all the $r$-sets in $[n]$ intersecting a given $t$-set in exactly one vertex, respectively.
We have
$|\cS^r_{n,t}|=\binom{n}{r}-\binom{n-t}{r}=\binom{n-1}{r-1}+\binom{n-2}{r-1}+\ldots+\binom{n-t}{r-1}$
and $|\cC^r_{n,t}|=t\binom{n-t}{r-1}.$
In particular, $$|\cS^r_{n,t}|,\, |\cC^r_{n,t}|\sim
t\binom{n}{r-1} \text{ as }n\to\infty
\text{ and $t$ and $r$ are fixed.} $$

For any given $r$-graph $\cH$, observe that $\cH\not\subseteq \cS^r_{n,\tau-1}$, where $\tau=\tau(\cH)$
and if  $\sigma(\cH)<\infty$ then $\cH\not\subseteq \cC^r_{n,\sigma-1}$, where $\sigma=\sigma(\cH)$.
Hence

\begin{proposition} \label{lower-bound-tau}
Let $r\geq 2$. Let $\cH$ be an $r$-graph with $\tau(\cH)=\tau$.
Then $$\ex(n,\cH)\geq |\cS^r_{n,\tau-1}|=
\binom{n-1}{r-1}+\binom{n-2}{r-1}+\ldots+\binom{n-\tau+1}{r-1}.  \qed$$  
\end{proposition}

\begin{proposition}\label{lower-bound-sigma}
Let $r\geq 2$. Let $\cH$ be an $r$-graph. Suppose
$\sigma(\cH)=\sigma<\infty$. Then
$$\ex(n,\cH)\geq  |\cC^r_{n,\sigma-1}|=(\sigma-1)\binom{n-\sigma+1}{r-1}.  \qed$$   
\end{proposition}

Recent works on hypergraph forests have identified many $r$-trees for which Proposition~\ref{lower-bound-sigma} is asymptotically tight, i.e.,
$\ex(n,\cH)=(\sigma(\cH)-1)\binom{n}{r-1}+o(n^{r-1})$.
We call such graphs {\it $\sigma$-tight}. There have been a few $r$-trees
for which equality holds in Proposition~\ref{lower-bound-tau} for sufficiently large $n$,
i.e.,  $\ex(n,\cH)=\binom{n}{r}-\binom{n-\tau(\cH)+1}{r}$ for all sufficiently large $n$.
We call such $r$-graphs {\it $\tau$-perfect}.
For example, $r$-uniform matchings are both $\sigma$-tight and $\tau$-perfect
(Erd\H{o}s~\cite{erdos-matching}).
See~\cite{BK, Furedi-trees, FJ-cycles, FJS,  KMV-path-cycle, KMV-trees}  for more recent works on $\sigma$-tight and $\tau$-perfect families.
Our main results in this paper generalize most of these recent results.
Cross-cuts were introduced by Frankl and F\"uredi~\cite{FF-exact}.
But they focused only on cases where $\sigma$ was small.
~\cite{Furedi-trees} 
was the first paper dealing with cases $\sigma\geq r-1$, followed shortly by
\cite{FJS}, \cite{FJ-cycles}, \cite{KMV-path-cycle}, and \cite{KMV-trees}.
Our work in this paper builds on  ~\cite{Furedi-trees},  ~\cite{FJ-cycles} , \cite{FJS},
unifies and substantially extends these results as well as others. In particular, we obtain
exact results for many graphs.


\section{Main results: asymptotic and stability} \label{s:main}

Let $r,k$ be positive integers where $r\geq k+1$.
An $r$-graph $\cH$ is {\it $k$-reducible} if each edge of
$\cH$ contains at least $k$ vertices of degree $1$.
If $\cH$ is $k$-reducible, we call the unique  $(r-k)$-graph
$\cG$ obtained by
$\cH$ by deleting $k$ degree $1$ vertices from each edge of $\cH$
the {\it $k$-reduction} of $\cH$. In general, the $k$-reduction $\cG$ of $\cH$ may be a multi-hypergraph. (For instance, the $(r-p)$-reduction of an $r$-uniform $s$-petal sunflower with kernel size $p$ consists of $s$ copies of one edge of size $p$).
We denote the underlying simple hypergraph of $\cG$ by $\cG'$.
We may view $\cH$ as being obtained from $\cG'$
by enlarging each edge $E$ to a given positive number $\mu(E)$ of new edges of size $r$
 by adding $k$ new vertices (called {\it expansion vertices}) per new edge in such a way that different new edges use disjoint sets of expansion vertices. If $\mu(E)=1$ for all
$E\in \cG'$, i.e. if $\cG=\cG'$, then we call $\cH$ a {\it simple expansion} of an $(r-k)$-graph; otherwise we call $\cH$ a {\it multi-expansion} of an $(r-k)$-graph.

\begin{theorem}{\bf (Asymptotic)} \label{main-asymp}
Let $\cH$ be an $r$-graph that is embeddable in a
$2$-reducible $r$-tree, where $r\geq 4$.
Let $\sigma=\sigma(\cH)$, define $\beta:=1/((r-2)(\sigma+1)+1)$. Then
\begin{equation}\label{eq:1}
(\sigma-1)\binom{n}{r-1}+O(n^{r-2})\leq \ex(n,\cH)\leq (\sigma-1)\binom{n}{r-1}+O(n^{r-1-\beta}).
 \end{equation}
In case of $\sigma=1$ we have an upper bound $\ex(n,\cH)= O(n^{r-2})$ (see \eqref{F0-upper}).
 \end{theorem}
 
By Theorem \ref{main-asymp}, all $r$-uniform multi-expansions of $(r-2)$-trees
are $\sigma$-tight when $r\geq 4$. On the other hand, not all
$r$-uniform multi-expansions of $(r-1)$-trees are $\sigma$-tight (for all $r\geq 2$).
So, Theorem~\ref{main-asymp}  is best possible in this sense. To construct a
non-$\sigma$-tight $r$-uniform multi-expansion of an $(r-1)$-tree,
one can take an $r$-graph $\cS$ of size $s$ such that $|\bigcap_{F\in \cH} F|=r-1$ (a sunflower of size $s$). Obviously $\sigma(\cS)=1$.
On the other hand, it is known  (R\"odl~\cite{rodl} and Keevash~\cite{keevash}) that whenever $s$ and $r$ are fixed and $n\to \infty$
 there are $\cS$-free families (called $P_{s-1}(n,r,r-1)$ packings) of size $\frac{s-1}{r}\binom{n}{r-1}+O(n^{r-2})$. So $\cS$ is not $\sigma$-tight. 
For all $r\geq 4$, Irwin and Jiang~\cite{IJ} also constructed infinitely many
$r$-uniform simple expansions $\cG$ of $(r-1)$-trees that are not $\sigma$-tight.
In fact, in their construction $\cG$, all but one edge of $\cG$ have two degree 1 vertices.
By contrast, Kostochka, Mubayi, and Verstra\"ete~\cite{KMV-trees} had earlier showed that
every $3$-uniform simple expansion of a $2$-tree is $\sigma$-tight.

\begin{theorem}{\bf (Structural Stability)} \label{main-stability}
Let $\cH$ be an $r$-graph that is embeddable in a
$2$-reducible $r$-tree, where $r\geq 4$.
Let $\sigma=\sigma(\cH)$, suppose $\sigma\geq 2$ and define $\beta:=1/((r-2)(\sigma+1)+1)$.
Suppose that  $\cF\subseteq \binom{[n]}{r}$ is $\cH$-free and $n\geq n(r,s)$, (where $n(r,s)$ is a function of $r$ and $s$ only).
If
 $$|\cF|\geq (\sigma-1)\binom{n}{r-1}-Kn^{r-1-\beta}$$
for $K\geq 0$, then there exists a set $A$ of $\sigma-1$ vertices such that
$$ |\cL_\cF(A)|\geq \binom{n}{r-1}-(r-1)(K+2s^2)n^{r-1-\beta}.
$$
Furthermore, all but at most
$$
 \left((\sigma-1) (r-1)(K+2s^2)+2s^2\right)n^{r-1-\beta}
$$
  members of $\cF$ meet $A$ in exactly one element.
\end{theorem}

Using the structural stability 
we 
  obtain exact results for certain critical graphs.

\begin{theorem} {\bf (Critical graphs)} \label{critical-exact}
Let $\cH$ be an $r$-graph that is embeddable in a $2$-reducible $r$-tree, where
$r\geq 4$.  Let $\sigma=\sigma(\cH)>1$.
Suppose $\cH$ contains an edge $F_0$ such that $\sigma(\cH\setminus F_0)=\sigma-1$.
Then there is a positive integer $n_0$ such that $\ex(n,\cH)\leq \binom{n}{r}-\binom{n-\sigma+1}{r}$ for all $n\geq n_0$.
If, in addition, $\tau(\cH)=\sigma(\cH)$, then for all $n\geq n_0$ we have
$$\ex(n,\cH)=\binom{n}{r}-\binom{n-\sigma+1}{r}.$$
\end{theorem}

For $r\geq 5$, Theorem~\ref{critical-exact} implies the exact results for $r$-uniform linear paths and cycles of odd length~\cite{FJ-cycles, FJS} and also the exact results~\cite{BK} on the disjoint union of linear paths and cycles at least one of which has odd length.
For general $2$-reducible $r$-trees, we can sharpen the error term in the
upper bound of Theorem~\ref{main-asymp} to $O(n^{r-2})$.

\begin{theorem} {\bf (Sharper estimates for trees)} \label{sharper}
Let $\cH$ be a $2$-reducible $r$-tree, where $r\geq 4$.
Let $\sigma=\sigma(\cH)$. Then $\ex(n,\cH)\leq (\sigma-1)\binom{n}{r-1}+O(n^{r-2})$.
\end{theorem}

For $(r-2)$-reducible $r$-trees, i.e., $r$-uniform multi-expansions of $2$-trees,
we can sharpen our estimates even further, which sometimes yields exact results.

\begin{theorem}[Sharper results on multi-expansions of $2$-trees] \label{2-tree-expansion}
Let $r\geq 4$ and $\cH$ an $(r-2)$-reducible $r$-tree with $\sigma(\cH)=t+1$.
Let $\pi$ be a tree-defining ordering of $\cH$ and $S$ be a minimum cross-cut of $\cH$. Let $w$ be the last vertex in $S$ that is included in $\pi$ and
$\cH_w$ the subgraph of $\cH$ consisting of all the edges containing $w$.
Suppose that $\cH_w$ is a linear star $\cL^r_{p}$ (i.e.,
  a sunflower with a single vertex in the kernel and $p$ petals).
Then there exists a positive integer $n_1$ such that for all $n\geq n_1$ we have
$$\ex(n,\cH)\leq \binom{n}{r}-\binom{n-t}{r}+\ex(n-t, \cL^r_{p}).$$
\end{theorem}

In many cases, Theorem \ref{2-tree-expansion} reduces the determination of
the Tur\'an number of an $(r-2)$-reducible $r$-tree to determination of the Tur\'an number  of a linear star $\cL^r_p$. For all $r\geq 5$ and $p\geq 2$,
Frankl and F\"uredi~\cite{FF-exact} determined  $\ex(n,\cL^r_p)$ asymptotically, showing that $\ex(n,\cL^r_p)=(\varphi(2,p)+o(1))\binom{n-2}{r-2}$, where $\varphi(2,p)$ is
the maximum size of a $2$-graph not containing a star of size $p$ or a matching of size $p$.
As determined by Abbott et al.~\cite{AHS},  $\varphi(2,p)=p(p-1)$ if $p$ is odd and that
$\varphi(2,p)=(p-1)^2+\frac{1}{2}(p-2)$ if $p$ is even.
Using the asymptotic result of Frankl and F\"uredi~\cite{FF-exact} and the stability method used in
this paper, Irwin and Jiang~\cite{IJ} were able to determine the exact value of $\ex(n,\cL^r_p)$
for all $r\geq 5$, $p\geq 2$ when $n$ is large.
Based on this, Theorem~\ref{2-tree-expansion} can then be used to obtain the exact value of $\ex(n,\cH)$ for many $(r-2)$-reducible $r$-trees $\cH$.

For instance, for $r\geq 4$ and large $n$, as was already obtained by Frankl~\cite{frankl1}, $\ex(n,\cL^r_2)=\binom{n-2}{r-2}$.
If $\cH$ is a linear path of even length or if $\cH$ is the disjoint union of linear  paths all of which have even length,
then $\cH_w=\cL^r_2$ 
and for large $n$ Theorem~\ref{2-tree-expansion} shows that $\ex(n,\cH)\leq \binom{n}{r}-\binom{n-t}{r}+\binom{n-t-2}{r-2}$, where $t=\sigma(\cH)-1$.
On the other hand, a trivial construction shows that $\ex(n,\cH)\geq \binom{n}{r}-\binom{n-t}{r}+\binom{n-t-2}{r-2}$.
So for $r\geq 4$, we immediately retrieve the even case of the exact results from~\cite{FJS} and~\cite{BK}.



\section{Lemmas on $r$-trees, partial $r$-trees, and cross-cuts}  \label{s:lemmas}

In this section, we develop a series of lemmas.
The following (easy) lemma was given in~\cite{Furedi-trees}.
\begin{proposition} \label{tight-tree}
Every $r$-tree $\cH$ is contained in a tight $r$-tree  $\cG$ with $V(\cG)=V(\cH)$.
Furthermore, a starting edge of $\cH$ can be used as a starting edge in $\cG$. \qed
\end{proposition}

\begin{lemma}[Tree embedding] \label{tree-embedding}
Let $\cH$ be an $r$-tree, where $r\geq 2$.
Let $E_1$ be a starting edge of $\cH$.
Let $\cF$ be an $r$-graph with $\delta_{r-1}(\cF)\geq |V(\cH)|-r+1$.
Then any mapping $f: V(E_1)\to V(\cF)$ such that $f(E_1)\in \cF$ can be extended to an embedding of $\cH$ in $\cF$.
\end{lemma}
\begin{proof}
By Proposition~\ref{tight-tree} we may assume that $\cH$ is a tight $r$-tree.
Let $E_1,\ldots, E_m$ be an ordering of the edges of $\cH$  defining $\cH$ as a tight $r$-tree.
For each $i\in [m]$ let $\cH_i=\{ E_j: j\leq i\}$ be the initial segment of $\cH$. Then $\cH_i$ is a tight $r$-tree.
We use induction on $i$ to show that $f$ can be extended to an embedding $f_i$ of $\cH_i$ in $\cF$. For the basis step, let $f_1=f$.
In general, let $2\leq i\leq |E(\cH)|$ and suppose $f$ can be extended to an embedding $f_{i-1}$ of $\cH_{i-1}$ in $\cF$. 
Let $E_{\alpha(i)}$ be a parent of $E_i$ and $D=E_i\cap E_{\alpha(i)}$. By definition, $|D|=r-1$.
Let $D'=f_{i-1}(D)$. 
Since $\delta_{r-1}(\cF)\geq |V(\cH)|-r+1$ we have $\deg_\cF(D')>|f_{i-1}\left(\bigcup \cH_{i-1}\right) |-|D'|$.
So one can find an edge $F$ in $\cF$ containing $D'$ such that $(F\setminus D')\cap f_{i-1}\left(\bigcup \cH_{i-1}\right)= \emptyset$.
Now $\cH'_i:=f_{i-1}\left(\cH_{i-1}\right)\cup F$ is a copy of $\cH_i$ in $\cF$. We extend $f_{i-1}$ to $f_i$ by mapping the single vertex in $E_i\setminus E_{\alpha(i)}$ to the single vertex in $F\setminus D'$. 
\end{proof}

\begin{proposition}[Embedding an expansion]   \label{expansion-embedding}
Let $r\geq 2$ be an integer.
Suppose that $\cG$ is an $r$-graph with $s$ vertices and $S$ is a set of degree $1$ vertices in $\cG$.
Let $\cF$ be an $r$-graph. If $\cG-S\subseteq \Ker_s(\cF)$, then $\cG\subseteq \cF$.
\end{proposition}
\begin{proof}
Easy from definitions.
Let $E_1,\ldots, E_m$ be the edges of $\cG$. For each $i$, let $D_i=E_i\setminus S$.
Note that $\{D_1,D_2,\ldots, D_m\}$ may be a multi-set.
Let $f$ be an embedding of $\cG-S$ into $\Ker_s(\cF)$.
Let $W=\bigcup_{i=1}^m f(D_i)$.
To obtain a copy of $\cG$ in $\cF$ it suffices to extend each $f(D_i)$, $1\leq i\leq m$, to some edge $F_i$
of $\cF$ containing $f(D_i)$ such that for $i=1,\dots, m$, $F_i\setminus f(D_i)$ are pairwise disjoint and that each $F_i\setminus f(D_i)$ is disjoint from $W$.
On can  define the appropriate $F_i$'s one by one using~\eqref{eq:extension}.
\end{proof}

\begin{proposition}[Trees and shadows]\label{tree-bound}
Let $\cH$ be an $r$-tree with $p$ vertices.
Let $\cF$ be an $r$-graph on $[n]$ not containing $\cH$.
Then $|\cF|\leq (p-r) |\partial_{r-1}(\cF)|$.
In particular, we have $\ex(n,\cH)\leq (p-r) \binom{n}{r-1}.$
\end{proposition}
\begin{proof} The second statement follows from the first since $|\partial_{r-1}(\cF)|\leq \binom{n}{r-1}$.
Suppose $|\cF|>(p-r) |\partial_{r-1}(\cF)|$.
We successively remove edges from $\cF$ that contain an $(r-1)$-set $D$ whose
degree becomes at most $p-r$ until no such edge remains.
Denote the remaining graph by $\cF'$.
Since at most $(p-r)$ edges were removed for each such $D$ and there are at most $|\partial_{r-1}(\cF)|$ such $D$,
$\cF'$ is nonempty. By definition, $\delta_{r-1}(\cF')\geq p-r+1$.
By Lemma~\ref{tree-embedding}, $\cH\subseteq \cF'\subseteq \cF$, a contradiction.
\end{proof}

Lemmas similar to Proposition \ref{tree-bound} were given in~\cite{Furedi-trees} and in~\cite{KMV-path-cycle}. For specific $r$-trees, more is known.
Katona~\cite{Katona} showed that for an intersecting family of $r$-sets $\cF$ (i.e. $\cF$ avoids a matching of size $2$) one has $|\cF|\leq |\partial_{r-1}(\cF)|$.
This was recently extended by Frankl~\cite{frankl-matching} who showed that if an $r$-graph $\cF$ avoids $\cM_s$ (a matching of size $s$)
 then $|\cF|\leq (s-1)|\partial_{r-1}(\cF)| $.

\begin{lemma}[The first edge containing a vertex]\label{first-edge}
Let $E_1,\ldots, E_m$ be an ordering of edges that defines a hypergraph tree $\cH$.
Fix an $i$, where $2\leq i\leq m$.
Let $x\in E_i\setminus E_{\alpha(i)}$.
Then $E_i$ is the first edge in the ordering that contains $x$.
Also, if $y\in E_{\alpha(i)}\setminus E_i$ then no edge of $\cH$ contains both $x$ and $y$.
\end{lemma}
\begin{proof}
By definition, $E_i\cap (\bigcup_{j<i} E_j)\subseteq E_{\alpha(i)}$. Since $x\notin E_{\alpha(i)}$, $x\notin  (\bigcup_{j<i} E_j)$. So $E_i$ is the first edge in the ordering that contains $x$.  Let $y\in E_{\alpha(i)}\setminus E_i$.
At the moment of $E_i$'s addition, $x,y$ are both in $\cH$ but there is no edge containing both $x,y$.  Suppose some later edge contains both $x$ and $y$. Let $E_j$ be
an earliest such edge, where $j>i$. Then $E_{\alpha(j)}$ must already contain both $x,y$, contradicting our $E_j$. So no edge of $\cH$ contains both $x$ and $y$.
\end{proof}

\begin{lemma} {\bf(Compression of trees)} \label{compression}
Let $\cH$ be a hypergraph tree with a defining ordering $E_1,\ldots, E_m$.
Fix one $i$, where $2\leq i\leq m$. Suppose $x\in E_i\setminus E_{\alpha(i)}$ and
$y\in E_{\alpha(i)}\setminus E_i$.
For each $j=1,\ldots, m$, let $E'_j=E_j\setminus \{x\}
\cup \{y\}$ if $x\in E_j$ and let $E'_j=E_j$ if $x\not\in E_j$.
Then the list $E'_1,\ldots, E'_m$  defines a hypergraph tree $\cH'$. 
\end{lemma}
\begin{proof}
By Lemma~\ref{first-edge}, $E_i$ is the first edge in the ordering
that contains $x$ and that no edge of $\cH$ contains both $x$ and $y$. 
Next we show that the multi-list $E'_1,\ldots, E'_m$ defines a hypergraph tree. It suffices to check that
\begin{equation} \label{ordering}
E'_j\cap (\bigcup_{1\leq k< j} E'_k)\subseteq E'_{\alpha(j)}
\end{equation}
holds for each $2\leq j\leq m$. By the definition of $\alpha(j)$, we have
\begin{equation} \label{ordering2}
E_j\cap (\bigcup_{1\leq k< j} E_k)\subseteq E_{\alpha(j)}.
\end{equation}

First suppose $j<i$. Since $E_i$ is the first edge in the ordering
that contains $x$,  we have $E'_\ell=E_\ell$ for all $\ell\leq j$.
So~\eqref{ordering} is the same as~\eqref{ordering2}, which holds.
Next, suppose $j=i$.
Then~\eqref{ordering} holds since $E'_i\cap
(\bigcup_{k<i} E'_k)\subseteq (E_i\cap (\bigcup_{k<i} E_k))\cup \{y\}
\subseteq E_{\alpha(i)}\cup \{y\}=E'_{\alpha(i)}$, where the last equality
follows from the fact that $E_{\alpha(i)}$ contains $y$ but not $x$.
Finally, suppose $j>i$.
Observe that the edges are unchanged outside $\{x,y\}$ and that the new edges do not contain $x$.
Suppose~\eqref{ordering} does not hold. Then we must have $y\in E'_j$ and $y\notin E'_{\alpha(j)}$. The latter implies $x\notin E_{\alpha(j)}$ and so $E'_{\alpha(j)}=E_{\alpha(j)}$. Thus, $y\notin E_{\alpha(j)}$. There are two subcases to check.

First, suppose $y\in E_j$. Then $y\in E_j\setminus E_{\alpha(j)}$.
By Lemma~\ref{first-edge}, $E_j$ is the first edge in the ordering that contains $y$. Since $y\in E_{\alpha(i)}$, we have $j\leq \alpha(i)<i$, contradicting $j>i$.
Next, suppose $y\notin E_j$. Since $y\in E'_j$, we have $x\in E_j$. But $x\notin E_{\alpha(j)}$. So $x\in E_j\setminus E_{\alpha(j)}$.
By Lemma~\ref{first-edge}, $E_j$ is the first edge in the ordering that contains $x$, contradicting $E_i$ being the first edge that contains $x$.

We have shown that the multi-list $E'_1,\ldots, E'_m$ satisfies~\eqref{ordering}.  
\end{proof}

\begin{corollary} {\bf (Smallest hosting tree)} \label{host-tree}
Let $\cH$ be an $r$-graph that is embeddable in an $r$-tree $\cT$ and let $V_1,\ldots, V_r$ be a good $r$-coloring of $V(\cT)$.
For each $i=1,\ldots, r$, define $X_i:=V_i\cap V(\cH)$.
Then there exists an $r$-tree $\cG$ containing $\cH$ satisfying that $V(\cG)=V(\cH)$
and that $X_1,\ldots, X_r$ is an $r$-coloring of $\cG$.
\end{corollary}
\begin{proof}
Let $E_1,\ldots, E_m$ be an ordering of the edges of $\cT$ that defines $\cT$ as a hypergraph tree.
If $V(\cT)=V(\cH)$ then we let $\cG=\cT$.
Otherwise, let $x\in V(\cT)\setminus V(\cH)$. Suppose $x\in V_r$.
Let $E_i$ be the first edge in the ordering that contains $x$. Then $x\in E_i\setminus E_{\alpha(i)}$.
Let $y$ be the unique vertex in $E_{\alpha(i)}\cap V_r$. Then $y\in E_{\alpha(i)}\setminus E_i$.
Let $\cT'$ denote the hypergraph tree obtained from $\cT$ by applying the compression procedure in Lemma~\ref{compression}. 
All the edges in $\cT$ that do not contain $x$ remain unchanged. Hence $\cH\subseteq \cT'$.
We repeat this procedure until we obtain $\cG$.
\end{proof}

For the next proposition, the reader should recall the definition of a trace,
given in Section~\ref{s:definitions}.
The following can be immediately verified using definitions.

\begin{proposition}[The trace of a tree] \label{deleting-vertices}
Let $\cG$ be a hypergraph tree and $S\subseteq V(G)$. Then $G|_S$ and $G-S$ are also
hypergraph trees.  \qed
\end{proposition}

\begin{proposition}[Subtrees through one vertex] \label{thru-one-vertex}
Let $\cH =\{ E_1,\ldots, E_m\}$ be an $r$-tree with this ordering, $r\geq 2$.
For each vertex $x$, let $\cH_x$ denote the subgraph consisting of the edges containing $x$.
Then $\cH_x$ is an $r$-tree.
\end{proposition}

\begin{proof} 
As usual, let $E_{\alpha(i)}$ the parent of $E_i$ in $\cH$ (for $i\geq 2$).
Order the edges of $\cH_x$ in the same way as they were in $\cH$.
Let $E_i$ be an edge in $\cH_x$ that is not the first edge of $\cH_x$.
Then $E_i$ is not the first edge in $\cH$ that contains $x$.
Hence its parent $E_{\alpha(i)}$ must already contain $x$.
So $E_{\alpha(i)}$ is also in $\cH_x$ and appears before $E_i$.
Let $E_j$ be any edge in $\cH_x$ appearing before $E_i$. Then it also appears before
$E_i$ in $\cH$ and hence $E_j\cap E_i\subseteq E_{\alpha(i)}$. This shows that
$E_{\alpha(i)}$ still serves as a parent of $E_i$ in $\cH_x$. 
\end{proof}

\begin{proposition}[Deleting a cross-cut] \label{deleting-a-cross-cut}
Let $r\geq 3$. Let $\cH$ be an $r$-graph embeddable in a $1$-reducible $r$-tree $\cT$.
Let $S$ be a cross-cut of $\cH$. Then $\cH-S$ is embeddable  in an $(r-1)$-tree on the same vertex set as $\cH-S$.
\end{proposition}
\begin{proof}
Starting with $S$, from each edge of $\cT$ that does not intersect $S$ we select a vertex of degree $1$ and add it to $S$.
Call the resulting set $S'$. Then $\cH-S\subseteq \cT-S'$.
By Lemma~\ref{deleting-vertices}, $\cT-S'$ is a hypergraph tree.
Each edge in $\cT-S'$ has size at most $r-1$.
We can round out those edges of $\cT-S'$ of size smaller than
$r-1$ to $(r-1)$-sets by adding new expansion vertices. Call the resulting
$(r-1)$-graph $\cT'$.  Then $\cT'$ is an $(r-1)$-tree that contains $\cH-S$.
By Corollary~\ref{host-tree}, there exists an $(r-1)$-tree $\cG$ containing $\cH-S$ on the same vertex set as $\cH-S$.
\end{proof}

Let us mention a potential difficulty in extending results on $r$-trees to those embeddable in $r$-trees.
Namely, if $\cH$ is an $r$-graph embeddable in an $r$-tree,
then $\cH$ may not always have a minimum cross-cut that can be
extended to a cross-cut of some $r$-tree $\cT$ that contains $\cH$.

\begin{example}[Cross-cuts of an $r$-graph embeddable in an $r$-tree might not extend] ${}$\\
Define the $4$-graph $\cH:=\{1abx_{a,b}, 1bc x_{b,c}, 1cdx_{c,d},
2aby_{a,b}, 2bc y_{b,c}, 2cdy_{c,d}\}$,
where $1$, $2$, $a$, $b$, $c$, $d$, $x_{a,b}$, $x_{b,c}$, $x_{c,d}$, $y_{a,b}$, $y_{b,c}$, $y_{c,d}$ are $12$ different vertices.
Then $\cH$ is embeddable in the $4$-tree $\{12ab, 12bc, 12cd \}\cup \cH$,
and $S=\{1,2\}$ is the unique minimum cross-cut of $\cH$.
But every $4$-tree $\cT$ that contains $\cH$ must have an edge containing both $1$ and $2$.  So there is no $4$-tree $\cT'$ with a crosscut $S'$ such that $\cH\subseteq\cT'$ and $S\subseteq S'$.
\qed
\end{example}

\noindent
By comparison, the case $\sigma=1$ is simpler.
\begin{proposition}\label{sigma1}
Let $\cH$ be an $r$-graph embeddable in an $r$-tree. 
Suppose $\sigma(\cH)=1$ and $\{x\}$ is a cross-cut of $\cH$.
Then there is an $r$-tree $\cG$ with $V(\cG)=V(\cH)$ such that $\{x\}$ is a cross-cut of $\cG$.
Moreover, if $\cT$ is $k$-reducible then $\cG$ can be $k$-reducible, too. 
\end{proposition}
\begin{proof}
Let $X_1,\ldots, X_r$ be an $r$-partition of $\cH$. Note that $\{x\}$ must by itself be one of the $X_i$'s.
The claim then follows immediately from Corollary~\ref{host-tree}.
\end{proof}

\begin{lemma}{\bf (Subtrees and detachable limbs)} \label{limb}
Let $\cH =\{ E_1,\ldots, E_m\}$ be an $r$-tree with this ordering, $r\geq 2$.
For each $x$, let $\cH_x$ denote the subtree consisting of edges containing $x$.
Suppose $S$ is a cross-cut of $\cH$, $|S|\geq 2$.
Then $\exists w\in S$ such that
$\cH'=\cH\setminus \cH_w$ is an $r$-tree. Furthermore, there exist an $E\in \cH_w$
and $F\in \cH'$ such that $E$ is a starting edge of $\cH_w$ and $V(\cH_w)\cap V(\cH')=E\cap F$.
\end{lemma}
\begin{proof}
Let $E_{\alpha(i)}$ denote a fixed parent of $E_i$ in $\cH$.
Let $w$ be the last vertex in $S$ that is included as we add edges of $\cH$ in
the order of $\pi:=\{ 1,2,3,\dots\}$.
We now verify that $\cH'=\cH\setminus \cH_w$ is an $r$-tree.
Let $\pi'$ be obtained from $\pi$ by deleting the edges of $\cH_w$ and keeping the relative order of the remaining edges.
Let $E_j$ be an edge in $\pi'$ that is not the first edge.
Since $S$ is a cross-cut of $\cH$ and $E_j\notin \cH_w$, $E_j$ contains exactly one vertex $x$ of $S$ and $x\neq w$.
If $E_{\alpha(j)}$ contains $w$ then $x\in E_j\setminus E_{\alpha(j)}$.
By Lemma~\ref{first-edge}, $E_j$ is the first edge in $\pi$ that contains $x$. But $E_{\alpha(j)}$
appears earlier than $E_j$. So $w$ is included earlier than $x$, contradicting our choice of $w$. So $w\notin E_{\alpha(j)}$, which means $E_{\alpha(j)}$ is also in $\cH'$ and
appears earlier than $E_j$. For any $E_\ell$ in $\pi'$ that appears earlier than $E_j$,
it also appears earlier than $E_j$ in $\pi$ and we have $E_\ell\cap E_j\subseteq E_{\alpha(j)}$ since $E_{\alpha(j)}$ is a parent of $E_j$ in $\pi$. This shows
that $E_{\alpha(j)}$ is still a parent of $E_j$ in $\pi'$. So $\pi'$ defines $\cH'$ as
an $r$-tree.

Let $E_k$ be the first edge in $\pi_w$. Then $E_{\alpha(k)}\in\cH'$.
We show that if $A\in \cH_w$ and $B\in \cH'$ then $A\cap B\subseteq E_k\cap E_{\alpha(k)}$.  For convenience for each edge $D$
in $\cH$ we let $\alpha(D)$ denote a fixed parent of it in $\pi$.
Observe that we have (a) $A\cap B\subseteq A\cap \alpha(B)$ if
$A$ appears before $B$ in $\pi$, and (b) $A\cap B\subseteq \alpha(A)\cap B$
if $B$ appears before $A$ in $\pi$.
Starting with $A\cap B$, we may obtain a superset by either replacing $B$ with $\alpha(B)$
or by replacing $A$ with $\alpha(A)$ depending on which of (a), (b) applies.
If $A\neq E_k$ then by earlier discussion, $\alpha(A)$ is still in $\cH_w$.
Also, since $B\in \cH'$, $\alpha(B)\in \cH'$ by earlier discussion. In particularly,
$\alpha(B)\notin\cH_w$. So we may repeatedly apply (a) or (b) in a way until
$A=E_k$ and $B$ is an edge appearing before $E_k$ in $\pi$. Then $A\cap B\subseteq
E_k\cap E_{\alpha(k)}$ holds. This proves the second part.
\end{proof}

One of the subtleties in this paper is the distinction between an $r$-graph embeddable in
an $r$-tree (like linear cycles) and an $r$-tree itself.
One of the difficulties in extending results on $r$-trees to those embeddable in $r$-trees is that
  the latter class is not known to possess the nice decomposition property described in the above Lemma~\ref{limb}.


\section{Reduction to centralized families} \label{s:8}

In this section, we use the delta system method to reduce the problem of embedding
$2$-reducible hypergraph trees and their subgraphs into a host graph $\cF$ to one where
$\cF$ belongs to a so-called centralized family.
The following lemma was developed using the delta system method,
and was used in earlier works. See~\cite{Furedi-Ozkahya, FJS, Furedi-trees, FJ-cycles} for some recent applications.
In particular,~\cite{FJS} and~\cite{FJ-cycles} contain some detailed discussions that are most relevant to what is needed in this paper.

Let $\cF$ be an $r$-partite $r$-graph with an $r$-partition
$(X_1,\ldots, X_r)$. So, each edge of $\cF$ contains
 exactly one element of each $X_i$.
Given $F\in \cF$ and $I\subseteq [r]$, let $F[I]=F\cap(\bigcup_{i\in I} X_i)$.
In other words, $F[I]$ is the projection of $F$ onto those parts indexed by $I$.
If $I=\{i\}$, we write $F[i]$ for $F[\{i\}]$. Let $\cF[I]=\{F[I]: F\in\cF\}$.

\begin{lemma}{\bf (The homogeneous subfamily lemma, see~\cite{furedi-1983})}  \label{homogeneous}
For any positive integers $s$ and $r$, there is a positive constant $c(r,s)$ such that for
{\bf every} family $\cF\subseteq \binom{[n]}{r}$ there exist $\cF^*\subseteq \cF$
with $|\cF^*|\geq c(r,s) |\cF|$ and some $\cJ\subseteq 2^{[r]}\setminus [r]$ (called {\it the intersection pattern})
such that
\begin{enumerate}
\item $\cF^*$ is $r$-partite, together with an $r$-partition $(X_1,\ldots, X_r)$.
\item
$\forall F\in \cF^*, \forall I\in \cJ, \deg^*_{\cF^*}(F[I])\geq s$ and $\forall I\notin \cJ \not\exists F'\in \cF^*$ satisfying $F\cap F'=F[I]$.
\item $\cJ$ is closed under intersection, i.e.,
 for all $I,I'\in \cJ$ we have $I\cap I'\in \cJ$ as well. \qed
\end{enumerate}
\end{lemma}

We will call $\cF^*$ (with the corresponding $\cJ$)  {\it $(r,s)$-homogeneous} with intersection pattern $\cJ$.

\begin{lemma}{\bf \cite{FF-exact}} \label{small-or-centralized}
Let $n\geq r\geq 3$.
Let $\cF^*\subseteq \binom{[n]}{r}$ be an $(r,s)$-homogeneous family with
a corresponding $r$-partition $(X_1,\ldots, X_r)$ and intersection pattern $\cJ\subseteq
2^{[r]}$. Then one of the following holds:\\
 {\rm (1)} $|\cF^*|\leq \binom{[n]}{r-2}$, or \\
 {\rm (2)} $\exists a,b\in  [r]$ such that $2^{[r]\setminus\{a,b\}}\subseteq \cJ$, or\\
 {\rm (3)} $\exists i\in [r]$, such that $\forall F\in \cF^*$, $\deg_{\cF^*}(F\setminus F[i])=1$
but $\forall I\subsetneq [r]$, with $i\in I$ we have $\deg^*_{\cF^*}(F[I])\geq s$.
\end{lemma}
\begin{proof}
If $\cJ$ contains all the $(r-1)$-subsets of $[r]$ then since $\cJ$ is closed under intersection,
we have $\cJ=2^{[r]}\setminus [r]$, in which case (2) holds trivially. Hence, we may
assume that there is at least one $(r-1)$-subset of $[r]$ not in $\cJ$. This implies that
there are proper subsets of $[r]$ that are not contained in any member of $\cJ$.
Among them let $D$ be one with minimum  size.
Suppose there exist $F,F'\in \cF^*, F\neq F'$, such that $F[D]= F'[D]$. Then $F[D]\subseteq F\cap F'=F[B]$ for some $D\subseteq B\subseteq [r]$.
By Lemma~\ref{homogeneous} item 2, $B\in \cJ$, which contradicts our assumption about $D$.
So $F[D]$ are different from for each $F\in \cF^*$.
Hence $|\cF^*|\leq \binom{n}{|D|}$.
If $|D|\leq r-2$ then (1) holds and we are done. So assume $|D|=r-1$.

Without loss of generality, suppose $[r]\setminus \{i\}\not \in \cJ$ for $i=1,\ldots, t$ and $[r]\setminus \{i\}\in \cJ$ for $i=t+1,\ldots, r$.
By our assumption $t\geq 1$. Suppose first that $t\geq 2$.
 For any $i,j\in [t], i\neq j$,
observe that $[r]\setminus \{i,j\}\in \cJ$, since otherwise $r\setminus \{i,j\}$ is
a $(r-2)$-set not contained in any member of $\cJ$, contradicting our assumption about $D$.
Let $I$ be any subset of $[r]\setminus \{1,2\}$.
Then $I$ can be written as the intersection of sets of the form $[r]\setminus \{i\}$ for $i\in \{t+1,\ldots, r\}$ and
  $[r]\setminus \{i,j\}$, for $i,j\in [t], i\neq j$. Since each set of one of these forms
 are in $\cJ$ and $\cJ$ is closed under intersection, $I\in \cJ$.
So we have $2^{[r]\setminus \{1,2\}}\subseteq \cJ$, and
 (2) holds.
Finally, assume $t=1$.  Let $I$ be any proper subset of $[r]$ containing $1$.
Then since $[r]\setminus \{i\}\in \cJ$ for $i=2,\ldots, r$ and $\cJ$ is closed under intersection, we have $I\in \cJ$.
By Lemma~\ref{homogeneous} item 2, $\forall F\in \cF^*$ we have $\deg^*_{\cF^*}(F[I])\geq s$.
By our assumption, $[r]\setminus \{1\}\notin \cJ$.
By Lemma~\ref{homogeneous} item 2, $\forall F\in \cF^*$ $F[[r]\setminus \{1\}]=F\setminus F[i]$ is contained only in $F$ and not in any other member of $\cF^*$. So (3) holds.
\end{proof}

\begin{definition}
{\rm
If $\cF^*\subseteq \binom{[n]}{r}$ is an $(r,s)$-homogeneous family with a corresponding $r$-partition $(X_1,\ldots, X_r)$ and intersection pattern $\cJ$ for which
Lemma~\ref{small-or-centralized} item 3 holds, then we say that $\cF^*$ is {\it homogeneously centralized} with threshold $s$.
We call $i\in [r]$ the {\it central element} of $\cF^*$.
For each $F\in \cF^*$, we let $c(F)=F[i]$ and call it the {\it central element} of $F$.
More generally, $\cF\subseteq \binom{[n]}{r}$ is a {\it centralized family} with threshold $s$
if each $F\in \cF$ contains an element $c(F)$ such that $\forall c(F)\in D\subsetneq F$
we have $\deg^*_{\cF}(D)\geq s$. (The choice of $c(F)$ may not be unique, but we fix one.)}
\end{definition}

\begin{remark}\label{centralized-remarks}
{\rm
Note the following distinction between homogeneously centralized families and
centralized families:
If $\cF$ is homogeneously centralized then $\forall F\in \cF$,
$\deg_{\cF}(F\setminus c(F))=1$ (see Lemma~\ref{small-or-centralized} item 3). However, if $\cF$ is simply centralized, then
this condition need not hold.
}
\end{remark}
\begin{lemma} \label{central-or-small}
Let $r\geq 4$. Let $\cH$ be an $r$-graph that is embeddable in a $2$-reducible $r$-tree. Suppose $\cH$ has $s$ vertices. Let $\cF\subseteq \binom{[n]}{r}$
be an $(r,s)$-homogeneous family with a corresponding $r$-partition $(X_1,\ldots, X_r)$
and intersection pattern $\cJ$. If $\cH\not\subseteq \cF$, then either $|\cF|\leq \binom{n}{r-2}$ or $\cF$ is homogeneously centralized.
\end{lemma}
\begin{proof}
By Lemma~\ref{small-or-centralized}, it suffices to rule out item 2. Suppose otherwise
that item 2 holds for $\cF$ and $\cJ$. So there exist $M\subseteq [r]$ with $|M|=r-2$ such that $2^M\subseteq \cJ$. By our assumption,
for all $I\subseteq M$ and all $F\in \cF$,  $F[I]\in \Ker_s(\cF)$.
In particular, we have $\cF[M]\subseteq \Ker_s(\cF)$.
Let $D\in\partial_{r-3}(\cF[M])$. Then $D=F[I]$ for some $F\in \cF$ and $I\subseteq M$ with $|I|=r-3$. By our assumption, there is an $s$-star in $\cF$ with kernel $D$.
The restriction of the members
of this $s$-star to $\bigcup_{i\in M} X_i$ are $s$ edges in $\cF[M]$ containing $D$.
This shows that $\delta_{r-3}(\cF[M])\geq s$.

Let $\cT$ be a $2$-reducible $r$-tree that contains $\cH$. Let $\cH^*$ be obtained from
$\cH$ by removing two degree $1$ vertices from each edge of $\cH$ and eliminating
duplicated edges and let $\cT^*$ be
obtained from $\cT$ by removing two degree $1$ vertices from each edge of $\cT$
and eliminating duplicated edges.
Then clearly $\cH^*$ and $\cT^*$ are both $(r-2)$-uniform and $\cH^*\subseteq \cT^*$. By Lemma~\ref{deleting-vertices}, $\cT^*$ is an $(r-2)$-tree.
So $\cH^*$ is embeddable in an $(r-2)$-tree. By Lemma~\ref{tight-tree} and Lemma~\ref{host-tree}, there exists a tight $(r-2)$-tree $\cG$ containing $\cH^*$ with $V(\cG)=V(\cH^*)$.
In particular $\cG$, has at most $s$ vertices. Since $\delta_{r-3}(\cF[M])\geq s$,
by Lemma~\ref{tree-embedding}, $\cF[M]\supseteq \cG$ and thus $\cF[M]$
contains a copy $\cH'$ of $\cH^*$. Since $\cF[M]\subseteq \Ker_s(\cF)$, each
edge of $\cH'$ has kernel degree at least $s$ in $\cF$. By Lemma~\ref{expansion-embedding}, $\cH\subseteq \cF$, contradicting  $\cH\not\subseteq \cF$.
\end{proof}

\begin{theorem}\label{partition} {\bf (The reduction theorem)}
Let $r\geq 4$. Let $\cH$ be an $r$-graph that is embeddable in a $2$-reducible $r$-tree.
Suppose $\cH$ has $s$ vertices. If $\cH\not\subseteq \cF$, then $\cF$ can be split into subfamilies $\cF^*$ and $\cF_0$ such that $\cF^*$ is
centralized with threshold $s$ and  $|\cF_0|\leq \frac{1}{c(r,s)}\binom{n}{r-2}$.
Further, if $\sigma(\cH)=1$ then $\cF^*=\emptyset$ and thus $|\cF|\leq \frac{1}{c(r,s)}\binom{n}{r-2}$.
\end{theorem}

\begin{proof}
First we apply Lemma~\ref{homogeneous} to $\cF$ to get an $(r,s)$-homogeneous subfamily $\cF_1$ with
intersection pattern $\cJ_1$ such that $|\cF_1|\geq c(r,s) |\cF|$.
By Lemma~\ref{central-or-small}, either $|\cF_1|\leq \binom{n}{r-2}$ or $\cF_1$ is centralized.
If $|\cF_1|\leq \binom{n}{r-2}$
then  we stop. Otherwise we apply Lemma~\ref{homogeneous}
again to $\cF\setminus \cF_1$ to get an $(r,s)$-homogeneous subfamily $\cF_2$ with intersection pattern $\cJ_2$ such
that $|\cF_2|\geq c(r,s)(|\cF|-|\cF_1)$. We continue like this until $|\cF_i|\leq \binom{n}{r-2}$.
Let $m$ be the smallest index $i$ such that $|\cF_i|\leq \binom{n}{r-2}$.
By our assumption, for each $i\in [m-1]$, $\cF_i$ is homogeneously centralized.
Let $\cF^*=\bigcup_{i=1}^{m-1} \cF_i$ and let $\cF_0=\cF\setminus \cF^*$. Then clearly $\cF^*$ is centralized with threshold $s$.
Also, by the algorithm, $|\cF_m|\geq c(r,s)|\cF_0|$ and hence $|\cF_0|\leq \frac{1}{c(r,s)}|\cF_m|\leq  \frac{1}{c(r,s)}\binom{n}{r-2}$.

Next, suppose $\sigma(\cH)=1$ with $\{x\}$ being a cross-cut of $\cH$. By Proposition \ref{sigma1}, there exists an $r$-tree
$\cG$ containing $\cH$ such that 
  $\{x\}$ is a cross-cut of $\cG$.
Suppose $\cF^*\neq \emptyset$. Then $\cF_1\neq \emptyset$. By our assumption,
$\cF_1$ is homogeneously centralized with threshold $s$. So there exists an $r$-partition
$X_1,\ldots, X_r$ of $\cF_1$ together with a central element $c\in \{1,\ldots, r\}$, 
such that $\forall F\in \cF_1$ and $c\in I\subsetneq [r]$, $\deg^*_{\cF_1}(F[I])\geq s$.
This allows us to greedily embed $\cG$ into $\cF_1$, contradicting $\cF_1$ being $\cH$-free.
\end{proof}

\section{Proof of the asymptotic in Theorem \ref{main-asymp}} \label{s:main-theorem-proof}

Suppose that $\cH$ has $s$ vertices.
Since $\cH\not\subseteq \cF$, by Theorem~\ref{partition}, $\cF$ can be split
into subfamilies $\cF^*$ and $\cF_0$ such that $\cF^*$ is centralized with threshold $s$
and
\begin{equation} \label{F0-upper}
|\cF_0|\leq \frac{1}{c(r,s)}\binom{n}{r-2}.
\end{equation}

If $\sigma=1$, then by  Theorem~\ref{partition},
$|\cF|=|\cF_0|\leq \frac{1}{c(r,s)}\binom{n}{r-2}$, which implies the upper bound
in Theorem~\ref{main-asymp}. Hence, for the rest of this proof, we suppose that $\sigma> 1$.

By the definition of a centralized family, for each $F\in \cF^*$, there is a central element $c(F)\in F$ such that for all proper subsets $D$ of $F$ containing $c(F)$ we have $\deg^*_{\cF^*}(D)\geq s$.  Let

\begin{equation} \label{epsilon-definition}
\ve:={\frac{r-2}{(\sigma+1)(r-2)+1}} \quad \mbox{  and } \quad h:=\ce{n^\ve}.
\end{equation}
We accomplish the proof of Theorem~\ref{main-asymp} in three steps.

\medskip
\noindent
{\bf Step 1.} $\exists W\subseteq [n]$ and a subfamily $\cF^*_1\subseteq \cF^*$
such that $|W|=\ce{n^\ve}$, $F\cap W=\{c(F)\}$ for $\forall F\in \cF^*_1$, and
\begin{equation}\label{eq:step1}
  |\cF^*\setminus \cF^*_1|\leq s^2n^{r-1-(\ve/(r-2))}+n^{r-2+2\ve}.
\end{equation}

\medskip

\noindent
{\it Proof  of Step 1.}
We partition $\cF^*$ according to $c(F)$.
For each $i\in [n]$,
let $$\cA_i=\{F\in \cF^*: c(F)=i\}, \quad \mbox{ and } \quad \cA'_i=\{F\setminus \{i\}: F\in \cA_i\}.$$
By~Lemma~\ref{small-or-centralized}~(2) we have that $D\cup \{i\}\in \Ker^{(r-1)}_s(\cF^*)$ for all $D\in \partial_{r-2}(\cA'_i)$. Hence
\begin{equation} \label{kb-lower}
|\Ker^{(r-1)}_s(\cF^*)|\geq \frac{1}{r-1}\sum_{i=1}^n |\partial_{r-2}(\cA'_i)|.
\end{equation}
Let $\cT$ be a $2$-reducible $r$-tree that contains $\cH$. Let $\cT^*$ be obtained
from $\cT$ by deleting a degree $1$ vertex from each edge and $\cH^*$ be
obtained from $\cH$ by deleting a degree $1$ vertex from each edge.
Then $\cH^*\subseteq \cT^*$ and by Lemma~\ref{deleting-vertices}, $\cT^*$ is
an $(r-1)$-tree. So $\cH^*$ is embeddable in an $(r-1)$-tree.
By Lemma~\ref{host-tree}, there exists an $(r-1)$-tree $\cG$ containing $\cH^*$
such that $V(\cG)=V(\cH^*)$. If $\Ker^{(r-1)}_s(\cF^*)$ contains a copy of $\cG$, then
it contains a copy of $\cH^*$ by Lemma~\ref{expansion-embedding}.
Since $\cH\subseteq \cF^*\subseteq \cF$, this is a contradiction. So, $\Ker^{(r-1)}_s(\cF^*)$ does not contain $\cG$.
By Lemma~\ref{tree-bound}, we have
\begin{equation} \label{kb-upper}
|\Ker^{(r-1)}_s(\cF^*)|\leq  s\binom{n}{r-2}.
\end{equation}
For each $i\in [n]$, let $x_i\geq r-2$ be the real such that
$|\partial_{r-2}(\cA'_i)|=\binom{x_i}{r-2}$, where without loss of generality
we may assume that $x_1\geq \ldots\geq x_n$.
By~\eqref{kb-lower} and~\eqref{kb-upper}, we have
\begin{equation} \label{xi-bound}
\sum_{i=1}^n \binom{x_i}{r-2}\leq s(r-1) \binom{n}{r-2}.
\end{equation}

Since $x_1\geq \ldots \geq x_n$,~\eqref{xi-bound} gives $\binom{x_h}{r-2} \leq \frac{sr}{h}\binom{n}{r-2}$.
Hence, $x_h-r+3\leq (sr/h)^{1/(r-2)} n$.
Kruskal-Katona theorem~\eqref{eq:KK1.1} implies that  $|\cA'_i|\leq \binom{x_i}{r-1}$ holds  $\forall i\in [n]$, since $|\partial_{r-2}(\cA'_i)|=\binom{x_i}{r-2}$.
Note that $|\cA_i|=|\cA'_i|$. We obtain
\begin{equation*} 
\sum_{i> h} |\cA_i|\leq \sum_{i=h+1}^n \binom{x_i}{r-1}
 \leq \frac{x_h-r+3}{r-1}\sum_{i=h+1}^n \binom{x_i}{r-2} \leq  (\frac{sr}{h})^\frac{1}{r-2} n  s\binom{n}{r-2}<s^2 n^{r-1-(\ve/(r-2))}.
\end{equation*}

Define $W:=[h]$.  Let $\cF_1=\{F\in \cF^*: c(F)\notin W\}$.
We have
\begin{equation*} 
|\cF_1|\leq s^2 n^{r-1-(\ve/(r-2))}.
\end{equation*}
Let $\cF_2=\{F: F\in \cF^*\setminus \cF_1, |F\cap W|\geq 2\}$. Then
\begin{equation*} 
|\cF_2|\leq \binom{|W|}{2}\binom{n-|W|}{r-2}\leq n^{r-2+2\ve}.
\end{equation*}
Let $\cF^*_1=\cF^*\setminus (\cF_1\cup \cF_2)$.
By definition,  $\forall F\in \cF^*_1 $, we have  $F\cap W=\{c(F)\}$.
The above two bounds imply  
$$|\cF^*\setminus \cF^*_1|=
|\cF_1\cup \cF_2|\leq s^2 n^{r-1-(\ve/(r-2))}+n^{r-2+2\ve}.$$
This completes Step 1. \qed

\medskip

Let $S$ be a cross-cut of $\cH$ with $|S|=\sigma=\sigma(\cH)$.
For the next claim, the reader should recall the definition of a common link graph from Section~\ref{s:definitions}.

\medskip
\noindent
{\bf Step 2.} For every $A\in \binom{W}{\sigma}$,
we have $|\cL_{\cF^*_1}(A)|\leq s \binom{n}{r-2}$ and
\begin{equation}\label{eq:step2b}
 |\bigcap_{x\in A} \partial_{r-2}(\cL_{\cF^*_1}(x))|\leq s \binom{n}{r-3}.
  \end{equation}

\medskip

\noindent
{\it Proof of Step 2.}
By Lemma~\ref{deleting-a-cross-cut},
there exists an $(r-1)$-tree $\cG$ containing $\cH-S$ with $V(\cG)=V(\cH-S)$.
In particular, $\cG$ has at most $s$ vertices and by Lemma~\ref{tree-bound},
$\ex(n,\cG)\leq s\binom{n}{r-2}$. Suppose there exists a $\sigma$-set $A$ in $W$
with $\cL_{\cF^*_1}(A)>s\binom{n}{r-2}$, then $\cL_{\cF^*_1}(A)$ contains
a copy of $\cH'$ of $\cH-S$. We then obtain a copy of $\cH$ in $\cF^*_1$ be
mapping $S$ to $A$, a contradiction.

Let us now select a degree $1$ vertex outside $S$ from each edge of $\cH$
and denote the resulting set $S'$. The set $S'$ is well-defined since each
edge of $\cH$ contains at least two degree $1$ vertices at most one of which
is in $S$. Observe that $S$ and $S'$ are two disjoint cross-cuts of $\cH$.
Let $\cH^*=\cH-S'$. By Lemma~\ref{deleting-a-cross-cut}, $\cH^*$
is embeddable in an $(r-1)$-tree on the same vertex set as $\cH^*$.
Clearly, $S$ is still a cross-cut of $\cH^*$.
Applying Lemma~\ref{deleting-a-cross-cut}
again, there exists an $(r-2)$-tree $\cG'$ containing $\cH^*-S$ on the
same vertex set as $\cH^*-S$. In particular, $\cG'$ has at most $s$ vertices
and hence $\ex(n,\cG')\leq s\binom{n}{r-3}$.

Suppose there exists a $\sigma$-set $A$ in $W$ such that $|\bigcap_{x\in A} \partial_{r-2}(\cL_{\cF^*_1}(x))|> s\binom{n}{r-3}$.
Then there exists a  copy $\cH'$ of $\cH^*-S$ in
$\cZ=\bigcap_{x\in A} \partial_{r-2}(\cL_{\cF^*_1}(x))$.  By embedding $S$ to $A$,
we see that $A\times \cZ$ contains a copy of $\cH^*$. By the definitions of $\cF^*_1$,
$A$, and $\cZ$, we have $A\times \cZ\subseteq \Ker^{(r-1)}_s(\cF^*)$. Hence $\cH^*\subseteq \Ker^{(r-1)}_s(\cF^*)$. Since $\cH^*=\cH-S'$,
by Lemma~\ref{expansion-embedding},  we get $\cH\subseteq \cF^*\subseteq \cF$, a contradiction.  \qed

\medskip

\noindent
{\bf Step 3.} $\exists \cF^*_2\subseteq \cF^*_1$ such that
\begin{equation}\label{eq:step3aa}
   \deg_{\cF^*_2}(F\setminus W)\leq \sigma-1
 \end{equation}
 holds for $\forall F\in \cF^*_2$ and
\begin{equation}\label{eq:step3a}|\cF^*_1\setminus \cF^*_2|\leq s n^{r-2+(\sigma+1) \ve}.
  \end{equation}
Furthermore, for each $(r-2)$-set $D\subseteq [n]\setminus W$,
\begin{equation}\label{eq:step3b}
|\left(\bigcup \{F\in \cF^*_2: D\subseteq F\}\right)\cap W|\leq \sigma-1.\end{equation}

\medskip

Note that~\eqref{eq:step3aa} gives
\begin{equation}\label{eq:step3c}
   |\cF_2^*|\leq (\sigma-1)\binom{n-|W|}{r-1}.
  \end{equation}

\noindent
{\it Proof of Step 3.} Note that $W$ is a cross-cut of $\cF^*_1$.
First, we clean out edges in $\cF^*_1$ that contain $(r-2)$-sets in $[n]\setminus W$
that lie in $\partial_{r-2}(\cL_{\cF^*_1}(x))$ for at least $\sigma$ different $x$ in $W$.
Formally, let
$$\cB:=\bigcup_{A\in \binom{W}{\sigma}} \left(\bigcap_{x\in A} \partial_{r-2}(\cL_{\cF^*_1}(x))\right).$$

By~\eqref{eq:step2b}, we have
\begin{equation*} 
|\cB|\leq \binom{|W|}{\sigma} s\binom{n}{r-3}<sn^{r-3+\sigma\ve}.
\end{equation*}
Let
$$\cF_3:=\{F\in \cF^*_1: \exists D\in \cB, D\subseteq F\}$$
Then
\begin{equation*} 
|\cF_3|\leq |\cB||W|n<sn^{r-2+(\sigma+1)\ve}.
\end{equation*}

Let $\cF^*_2:=\cF^*_1\setminus \cF_3$.
Then $|\cF^*_1\setminus \cF^*_2|=|\cF_3|\leq
sn^{r-2+(\sigma+1)\ve}$. For each $(r-2)$-set $D$ in $[n]\setminus W$ that
is contained in an edge of $\cF^*_2$, we have $D\notin \cB$.
So $|\left(\bigcup \{F\in \cF^*_2: D\subseteq F\}\right)\cap W| \leq \sigma-1$.
This also implies that $\deg_{\cF^*_2}(F\setminus W)\leq \sigma-1$ holds for $\forall F\in \cF^*_2$.
This completes Step 3. \qed

\medskip

The obvious identity
$$
  |\cF|=|\cF_0|+ |\cF^*\setminus \cF_1^*|+|\cF_1^*\setminus \cF_2^*|+|\cF_2^*|
  $$
  together with~\eqref{F0-upper},
 \eqref{eq:step1},~\eqref{eq:step3a}, and~\eqref{eq:step3c} yield
\begin{multline} \label{general-upper}
|\cF|\leq  \frac{1}{c(r,s)}\binom{n}{r-2} + (s^2n^{r-1-(\ve/(r-2))} + n^{r-2+2\ve})+ sn^{r-2+(\sigma+1)\ve}+
    |\cF_2^*|\\
    \leq  |\cF_2^*|+ (s^2+ s+2)n^{r-2+(\sigma+1)\ve}
 \leq (\sigma-1)\binom{n}{r-1}+ 2s^2n^{r-1-\frac{1}{(\sigma+1)(r-2)+1}},
\end{multline}
for $n\geq n(r,s)$, where $n(r,s)$ is some function of $r$ and $s$.
This completes the proof of Theorem~\ref{main-asymp}. \qed

\section{Proof of the stability in Theorem \ref{main-stability}} \label{s:main-theorem-proof2}

This is a continuation of the previous section.
Recall that we now assume 
                       $|\cF|\geq (\sigma-1)\binom{n}{r-1} -Kn^{r-1-\beta}$.
We already have from~\eqref{general-upper} and from the lower bound constraint for $|\cF|$ that
 for $n> n(r,s)$
\begin{equation} \label{F*2-lower}
|\cF^*_2|\geq |\cF|-2s^2n^{r-1-\beta}\geq (\sigma-1)\binom{n}{r-1} -(K+2s^2)n^{r-1-\beta}.
\end{equation}
Let
  $$ \cF^*_3=\{F\in \cF^*_2: \deg_{\cF^*_2}(F\setminus W)= \sigma-1\}.$$
Then $\forall F\in \cF^*_2\setminus \cF^*_3$ we have $\deg_{\cF^*_2}(F\setminus W)\leq \sigma-2$.
Counting the degrees of the $(r-1)$-sets of $[n]\setminus W$ in $\cF_2^*$ we obtain that
\begin{equation*} 
\frac{|\cF^*_3|}{\sigma-1}  + (\sigma-2)\binom{n-|W|}{r-1} \geq |\cF^*_2|.
\end{equation*}
This and~\eqref{F*2-lower} give
\begin{equation} \label{F*3-lower}
\frac{|\cF^*_3|}{\sigma-1} \geq \binom{n}{r-1} -(K+2s^2)n^{r-1-\beta}.
\end{equation}

Let $A_1,A_2,\ldots, A_m$ be all the subsets of $W$ of size $\sigma-1$ with $\cL_{\cF^*_3}(A_i)\neq \emptyset$.
Then
$$\cF^*_3=\bigcup_{1\leq i\leq m} (A_i\times \cL_{\cF^*_3}(A_i)).
  $$
By~\eqref{eq:step3b} and that fact that $\cF^*_3\subseteq \cF^*_2$,
\begin{equation} \label{disjoint-shadows}
\forall i,j\in [m], i\neq j,
\partial_{r-2}(\cL_{\cF^*_3}(A_i))\cap
\partial_{r-2}(\cL_{\cF^*_3}(A_j))=\emptyset.
\end{equation}

For each $i\in [m]$, let
$y_i\geq r-1$ denote the real such that $|\cL_{\cF^*_3}(A_i))|=\binom{y_i}{r-1}$.
Without loss of generality, we may assume that $y_1\geq y_2\geq \cdots \geq y_m$.
For each $i\in [m]$, by the Kruskal-Katona theorem~\eqref{eq:KK1.1}, we have
$|\partial_{r-2}(\cL_{\cF^*_3}(A_i))|\geq\binom{y_i}{r-2}$.
By~\eqref{disjoint-shadows}, we have
\begin{equation*} 
\sum_{1\leq i\leq m} \binom{y_i}{r-2}\leq \sum_{i=1}^m |\partial_{r-2}(\cL_{\cF^*_3}(A_i))|\leq \binom{n}{r-2}.
\end{equation*}
The disjointness of the $\cL_{\cF^*_3}(A_i)$'s imply $|\cF^*_3|=(\sigma-1)\sum_{i=1}^m |\cL_{\cF^*_3}(A_i)|$.
Use this, then the fact that $y_i\leq y_1$ for all $i$ and then the last displayed inequality.
We obtain
\begin{equation*} 
\frac{|\cF^*_3|}{\sigma-1}=\sum_{i=1}^m |\cL_{\cF^*_3}(A_i)|=\sum_{i=1}^m \binom{y_i}{r-1}
\leq \frac{y_1-r+2}{r-1}\sum_{i=1}^m \binom{y_i}{r-2} \leq \frac{y_1-r+2}{r-1} \binom{n}{r-2}.
\end{equation*}
Compare this to the lower bound~\eqref{F*3-lower}. We get
\begin{equation*} 
 \frac{y_1-r+2}{r-1} \binom{n}{r-2}\geq \frac{|\cF^*_3|}{\sigma-1}
    \geq \frac{n-r+2}{r-1}\binom{n}{r-2} -(K+2s^2)n^{r-1-\beta} .
\end{equation*}
Hence
$$(r-1) (K+2s^2)n^{r-1-\beta}  \geq (n-y_1)\binom{n}{r-2}.$$
Take $A=A_1$. We have
\begin{multline*}
|\cL_\cF(A)|\geq |\cL_{\cF^*_3}(A_1)|=\binom{y_1}{r-1}
\geq \binom{n}{r-1}-(n-y_1)\binom{n}{r-2}\\
 \geq \binom{n}{r-1}-(r-1)(K+2s^2)n^{r-1-\beta}.
 \end{multline*}
This, together with \eqref{general-upper}, also yields
$$|\cF\setminus (A \times\cL_\cF(A))|\leq
\left((\sigma-1) (r-1)(K+2s^2)+2s^2\right)n^{r-1-\beta}. \qed  $$


\section{Structures of near extremal families}
\label{s:r-partite-lemma}

To prove Theorems~\ref{critical-exact}, \ref{sharper}, \ref{2-tree-expansion}, we analyze the structure of near extremal families.

\begin{lemma}[Missing edges vs. non-$\cM$ edges]\label{thin-versus-missing}
Let $\cM$ be an $r$-graph with $m$ edges, $m\geq 2$.
Let $\cG$ be an $r$-graph on $[n]$, $\overline{\cG}$ its complement.
Let $\cG_0$ be the subgraph of $\cG$ consisting of the edges of $\cG$ that do not lie in any copy of $\cM$.
Then $|\cG_0|\leq (m-1)|\overline{\cG}|$.
\end{lemma}
\begin{proof}
Let $\cM_1,\ldots, \cM_h$  be all the labelled copies of $\cM$
on $[n]$. By symmetry, each $r$-set in $[n]$ lies in the
same number $t$ of these copies. If some $\cM_i$ contains an edge of $\cG_0$
then not all of its edges are in $\cG$ and so it contains an edge of $\overline{\cG}$.
Let $\mu$ be the number of triples $(e,M,f)$, where $e\in \cG_0, M\in \{\cM_1,\ldots, \cM_h\}, f\in \overline{\cG}$, and $e,f\in M$. Then $t|\cG_0|\leq \mu$ and $\mu\leq t(m-1)|\overline{\cG}|$. 
\end{proof}

Let $\cK^p_p(s)$ denote the complete $p$-partite $p$-graphs with $s$
vertices in each part. By a well-known result of Erd\H{o}s~\cite{erdos-complete-partite}, $\ex(n,\cK^p_p(s))\leq c_1(p,s) n^{p-(1/s^{p-1})}$ for all $n$
 where $c_1$ depends only on $p$ and $s$.

\begin{lemma} \label{small-bad}
Let $s,\, p\geq 2$ be fixed.
Let $n\geq n_2(p,s)$ be sufficiently large.
Let $\cG,\cD\subseteq \binom{[n]}{p}$, $\cD\neq \emptyset$.
Suppose 
   that
\begin{equation}\label{eq:91}
  |\cG|\geq \binom{n}{p}-n^{p-(1/s^{p-1})}.
  \end{equation}
Let $\cG^*\subseteq \cG$ consist of all edges of $\cG$ that lie
in copies of $K^p_p(s)$. Suppose $\partial_{p-1}(\cG^*)\cap \partial_{p-1}(\cD)=\emptyset$.
Then $|\cD|\leq c_2 n^{-1/((p-1)s^{p-1})}|\overline{\cG}|$,
for some positive constant $c_2:=c_2(p,s)$.
\end{lemma}
\begin{proof}
Let $\cG_0=\cG\setminus \cG^*$. By Erd\H{o}s' theorem,
$|\cG_0|<c_1n^{p-(1/s^{p-1})}$.  By Lemma \ref{thin-versus-missing},
$|\cG_0|\leq (s^p-1)|\overline{G}|$.
Hence $|\overline{\cG^*}|\leq s^{p}|\overline{\cG}|$.
This and~\eqref{eq:91} gives
\begin{equation}\label{eq:G*}
|{\cG^*}|\geq \binom{n}{p} -c_3 n^{p-(1/s^{p-1})},
\end{equation}
where the positive constant $c_3:=c_3(p,s)$ depends only on $p$ and $s$.

Let $x,y\geq p$ be positive reals such that
 $|\partial_{p-1}(\cG^*)|=\binom{x}{p-1}$ and $|\partial_{p-1}(\cD)|=\binom{y}{p-1}$.
The Kruskal-Katona theorem~\eqref{eq:KK1.1} implies that
\begin{equation} \label{eq:DG-upper}
|\cD|\leq \binom{y}{p}\enskip \text{ and }\enskip |\cG^*|\leq \binom{x}{p}.
\end{equation}
The inequality~\eqref{eq:G*} gives
\begin{equation} \label{eq:x-lower}
x\geq n-c_4 n^{1-(1/s^{p-1})}
\end{equation}
for some positive constant $c_4=c_4(p,s)$.

 Since
$\partial_{p-1}(\cD)\cap \partial_{p-1}(\cG^*)=\emptyset$, we have
\begin{equation}\label{eq:shadow-sum}
\binom{y}{p-1}+\binom{x}{p-1}\leq \binom{n}{p-1}.
\end{equation}
This implies
$\binom{y}{p-1}\leq \binom{n}{p-1}-\binom{x}{p-1}\leq (n-x)\binom{n}{p-2}$.
Using this and~\eqref{eq:x-lower} we get
\begin{equation}\label{eq:y-upper}
  y\leq c_5 n^{1-1/((p-1)s^{p-1})}
  \end{equation}
where $c_5:=c_5(p,s)$ and $n$ is large enough ($n> n_5(p,s)$).

Rewrite $\binom{n}{p-1}$ as $\frac{p}{n-p+1}\binom{n}{p}$ and multiply~\eqref{eq:shadow-sum} by $(x-p+1)/p$. We obtain
\begin{equation*}
 \frac{x-p+1}{y-p+1} \binom{y}{p}+\binom{x}{p}\leq  \frac{x-p+1}{n-p+1}\binom{n}{p} \leq \binom{n}{p}.
\end{equation*}
By~\eqref{eq:DG-upper}, we have
$$\frac{x-p+1}{y-p+1}|\cD|+|\cG^*|\leq   \binom{n}{p}.$$
Hence, by~\eqref{eq:x-lower} and~\eqref{eq:y-upper}, and using
$n$ being large enough, we have
\begin{equation*}
|\cD|\leq \frac{y-p+1}{x-p+1}|\overline{\cG^*}|\leq \frac{y-p+1}{x-p+1}s^{p}|\overline{\cG}| \leq c_2 n^{-1/((p-1)s^{p-1})}|\overline{\cG}|,
\end{equation*}
where $c_2:=c_2(p,s)$ depends only on $p$ and $s$.
\end{proof}

\begin{lemma} \label{fine-structure}
Let $\cH$ be an $r$-graph embeddable in a $2$-reducible $r$-tree.
Suppose $\cH$ has $s$ vertices and $\sigma(\cH)=\sigma\geq 2$.
Let $n\geq n_3(r,s)\geq n(r,s)$ be sufficiently large.
Let $\cF\subseteq \binom{[n]}{r}$ such that $\cH\not\subseteq \cF$ and
that $|\cF|\geq (\sigma-1)\binom{n}{r-1}-n^{r-1-\beta}$.
Let $A$ be be a $(\sigma-1)$-set guaranteed by Theorem \ref{main-stability} that satisfies
\begin{equation} \label{large-link}
\cL_\cF(A)\geq \binom{n}{r-1}-(r-1)(1+2s^2)n^{r-1-\beta}, \quad \mbox{ where } \beta=((r-2)(\sigma+1)+1)^{-1}.
\end{equation}
Let $\cL^*\subseteq \cL_\cF(A)$ consist of all edges of $\cL_\cF(A)$ that lie in
copies of $\cK^{(r-1)}_{r-1}(s)$.
Let
\begin{equation*} 
  \cF_A:=\{F\in \cF: F\cap A\neq \emptyset\},\quad \cS_A:=\{F\in \binom{[n]}{r}: F\cap A\neq \emptyset\}, \mbox { and } \cB\subseteq \cF\setminus \cF_A.
  \end{equation*}
Suppose either $\cB=\emptyset$ or there is an  $(r-1)$-graph $\cD$  on $[n]\setminus A$ satisfying
\begin{equation}\label{eq:D-requirements}
\partial_{r-2}(\cD)\cap \partial_{r-2}(\cL^*)=\emptyset \,\text{ and }\,
 |\cD|\geq \gamma |\cB|
 \end{equation}
for some $\gamma>0$.
Then $|\cF_A\cup \cB|\leq \binom{n}{r}-\binom{n-\sigma+1}{r}$ holds for $n> n_4:=n_4(r,s,\gamma)$.
\end{lemma}
\begin{proof}
Note that $\cL_\cF(A)$ is an $(r-1)$-graph on $[n]\setminus A$ and for sufficiently large $n\geq n_3$, we have
$$\cL_\cF(A)\geq \binom{n-\sigma+1}{r-1}-(n-\sigma+1)^{r-1-(1/s^{r-2})}.$$
Let $\overline{\cL}=\binom{[n]\setminus A}{r-1}\setminus \cL_\cF(A)$, i.e. $\overline{\cL}$
is the complement of $\cL_\cF(A)$ on $[n]\setminus A$.
Since $\cD$ is an $(r-1)$-graph on $[n]\setminus A$ satisfying $\partial_{r-2}(\cD)\cap \partial_{r-2}(\cL^*)=\emptyset$,
 we can apply Lemma~\ref{small-bad} with $\cG:=\cL_\cF(A)$.
We obtain
$$|\cD|\leq c_2 (n-\sigma+1)^{-1/((r-2)s^{r-2})}|\overline{\cL}|.$$
If $\cB=\emptyset$ then the lemma holds trivially. So assume $\cB\neq \emptyset$.
By our assumption about $\cB$ and $\cD$,
\begin{equation}\label{bad-upper}
|\cB|\leq \frac{1}{\gamma}|\cD|\leq \frac{c_2}{\gamma} (n-\sigma+1)^{-1/((r-2)s^{r-2})}|\overline{\cL}|.
\end{equation}
Each $(r-1)$-set in $\overline{\cL}$ contributes at least one edge to $\cS_A\setminus \cF_A$. So,
\begin{equation*}
|\cS_A\setminus \cF_A|\geq |\overline{\cL}|.
\end{equation*}
Now,
\begin{equation*} 
|\cF_A\cup \cB|\leq |\cS_A| - |\overline{\cL}|+|\cB|.
  \end{equation*}
By~\eqref{bad-upper}, we have $|\cB|\leq |\overline{\cL}|$
for sufficiently large $n$. Hence $|\cF_A\cup \cB|\leq |\cS_A|=\binom{n}{r}-\binom{n-\sigma+1}{r}$
for sufficiently large $n$.
\end{proof}


\section{Proof of Theorem~\ref{critical-exact} on critical edges}
\label{s:critical-exact}

Let $\cH$ be the $r$-graph on $s$ vertices embeddable in a $2$-reducible $r$-tree in~Theorem~\ref{critical-exact}
and let $\cF\subseteq \binom{[n]}{r}$ such that $\cH\not \subseteq \cF$ and
such that $n\geq n_4(r,s,\gamma)$, where $n_4(r,s,\gamma)$ is specified in Lemma \ref{fine-structure} with $\gamma=1/s$.
We may assume that $|\cF|\geq \binom{n}{r}-\binom{n-\sigma+1}{r}$, since otherwise we are done.
By Theorem \ref{main-stability}, there exists a $(\sigma-1)$-set $A$ that satisfies~\eqref{large-link}.
Define $\cL^*, \cF_A, \cS_A$ as in Lemma \ref{fine-structure}, and
define $\cB:=\cF\setminus \cF_A$. If $\cB=\emptyset$, then we are done.
Suppose $\cB\neq \emptyset$. If we can show that there exists an $(r-1)$-graph
$\cD$ on $[n]\setminus A$ satisfying~\eqref{eq:D-requirements}, i.e. $\partial_{r-2}(\cD)\cap \partial_{r-2}(\cL^*)
=\emptyset$ and $|\cD|\geq \gamma|\cB|$, then by Lemma \ref{fine-structure},
$|\cF|=|\cF_A\cup \cB|\leq \binom{n}{r}-\binom{n-\sigma+1}{r}$, and we are done.

Towards that goal, let $\cD:=\partial_{r-1}(\cB)$
  and $\gamma:=1/s$.   
Since $\cH\not\subseteq \cB$, by
Lemma~\ref{tree-bound},
$|\cB|\leq s|\partial_{r-1}(\cB)|=s|\cD|$, and hence $|\cD|\geq (1/s)|\cB|$, as desired.
By our assumption about $\cH$, $\cH$ contains an edge $F_0$ such that $\sigma(\cH\setminus F_0)=\sigma-1$.
Let $W=F_0\cap V(\cH\setminus F_0)$.
Since $F_0$ contains at least two degree $1$ vertices,  $|W|\leq r-2$.
By our assumption, $\cH'=\cH\setminus F_0$ has a cross-cut $S'$ of size $\sigma-1$.
Let $x$ be a degree $1$ vertex in $F_0\setminus W$.
Then $S=S'\cup \{x\}$ is a cross-cut of $\cH$.
By Proposition~\ref{deleting-a-cross-cut}, $\cH-S$ is embeddable in an $(r-1)$-tree and hence it is $(r-1)$-partite. Since $\cH'-S'\subseteq \cH-S$, $\cH'-S'$ is an $(r-1)$-partite $(r-1)$-graph.

Suppose for contradiction that $\exists U\in \partial_{r-2}(\cL^*)\cap \partial_{r-2}(\cB)$.
Let $E$ be an edge of $\cB$ that contains $U$.
Let $M$ be an edge of $\cL^*$ that contains $U$.
By our assumption, there exists a copy $\cK$ of $\cK^{(r-1)}_{(r-1)}(s)$ in $\cL^*$ containing $M$.
Let $U'$ be a subset of $U$ of size exactly $|W|$.
Since $\cH'-S'$ is an $(r-1)$-partite $(r-1)$-graph on fewer than $s$ vertices,
we can easily find a mapping $f$ of
$\cH'-S'$ into $\cK$ such that $W$ is mapped onto $U'$ and such that $f(\cH'-S')$
does not contain any vertex of $E\setminus U'$.
Now $(A\times f(\cH'-S'))\cup E\subseteq \cF$ contains a copy of $\cH$, a contradiction.
Hence $\partial_{r-2}(\cD)\cap \partial_{r-2}(\cL^*)=\emptyset$. This completes the proof. \qed


\section{Proof of Theorem~\ref{sharper}, sharper error term for $r$-trees}
\label{s:sharper}

Let $\cH$ be the $2$-reducible $r$-tree on $s$ vertices in Theorem~\ref{sharper}, $\sigma=\sigma(\cH)$,
and let $\cF\subseteq \binom{[n]}{r}$ such that $\cH\not \subseteq \cF$.
We will show that
\begin{equation}\label{eq:11}
 |\cF|\leq \binom{n}{r}-\binom{n-\sigma+1}{r}+\frac{1}{c(r,s)}\binom{n}{r-2}
  \end{equation}
for sufficiently large $n$, where $c(r,s)$ is the constant in Theorem~\ref{homogeneous}.
We may assume that $|\cF|\geq\binom{n}{r}-\binom{n-\sigma+1}{r}+\frac{1}{c(r,s)}\binom{n}{r-2}$, since otherwise there is nothing to prove.

By Lemma~\ref{fine-structure} there exists a $(\sigma-1)$-set $A$ that satisfies ~\eqref{large-link}.
Define $\cL^*, \cF_A$, and $\cS_A$ as in Lemma \ref{fine-structure} and
define $\cB:=\cF\setminus \cF_A$.
By Lemma~\ref{homogeneous}, there exists an $(r,s)$-homogeneous subfamily $\cB^*$ of $\cB$ with $|\cB^*|\geq c(r,s)|\cB|$. By Lemma~\ref{central-or-small},
either $|\cB^*|$ is homogeneously centralized with threshold $s$ or $|\cB^*|\leq \binom{n}{r-2}$.
If $|\cB|\leq \frac{1}{c(r,s)}\binom{n}{r-2}$, then since $\cF=\cF_A\cup \cB\subseteq \cS_A\cup \cB$, ~\eqref{eq:11} already holds.
Hence,  we may assume that $\cB^*$ is homogeneously centralized with threshold $s$.

For each $F\in \cB^*$ as usual let $c(F)$ denote the central element of $F$.
By the remarks before Lemma~\ref{central-or-small}, the
kernel degree of $D$ is at least $s$ in $\cB^*$ for all proper subsets $D$ of $F$ containing $c(F)$.
Furthermore, $F\setminus c(F)$ is contained in precisely one edge, namely $F$, of $\cB^*$.
Let $X_1,\ldots, X_r$ be an associated partition of $\cB^*$.
Without loss of generality we may assume that $\forall F\in \cB^*$, $F\cap X_1=\{c(F)\}$.
For each $x\in X_1$, let
$\cB_x=\{F\in \cB^*: c(F)=x\}$ and $\cD_x=\{F\setminus x: F\in \cB^*\}.$
The families $\cB_x$'s partition $\cB^*$.
Define
$\cD:=\bigcup_{x\in X_1} \cD_x$.

\medskip\noindent
 {\bf Claim~\ref{s:sharper}.1.}\quad $\partial_{r-2}(\cL^*)\cap \partial_{r-2}(\cD)=\emptyset$.

 \medskip \noindent {\it Proof of Claim~\ref{s:sharper}.1.}\quad
 Let $S$ be a minimum cross-cut.
By Lemma~\ref{limb}, there exists a $w\in S$ such that $\cH_w$ and $\cH'=\cH'\setminus \cH_w$ are $r$-trees.
Furthermore, there are $E\in \cH_w$ and $F\in \cH'$  such that $E$ is a starting edge of $\cH_w$ and $V(\cH')\cap V(\cH_w)= E\cap F$. Define $D:=E\cap F$.
Since $E$ has at least two degree $1$ vertices, $|D|\leq r-2$.

It suffices to show that $\forall x\in X_1$,
 $\partial_{r-2}(\cL^*)\cap \partial_{r-2} (\cD_x)=\emptyset$.
Suppose, on the contrary, that for some $x\in X_1$, $\exists M\in \partial_{r-2}(\cL^*)\cap \partial_{r-2}(\cD_x)$. Let
$M'$ be a subset of $M$ of size $|D|$. Let $E'$ be an edge of $\cD_x$ that contains
$M$ and thus contains $M'$.  Let $f$ be a mapping that maps $D$ onto $M'$.
Since $\cH_w-\{w\}$ is an $(r-1)$-tree on fewer than $s$ vertices,
by Lemma~\ref{first-edge}, $f$ can be extended to an embedding of $\cH_w-\{w\}$
into $\cD_x$. Let $E''$ be any edge of $\cL^*$ that contains $M'$. By the definition
of $\cL^*$, there exists a copy $\cK$ of $\cK^{(r-1)}_{r-1}(s)$ in $\cL^*$ that contains
$E''$. Let $S'=S\setminus \{w\}$.
 Since $\cH'-S'$ is $(r-1)$-partite and $\cH$ has $s$ vertices,
we can find a mapping $g$ of $\cH'-S'$ into $\cK$ that agrees with $f$ on
vertices in $D$ and such that $g(\cH'-S')\setminus M'$ is disjoint from
$f(\cH_w-\{w\})$. Now since $g(\cH'-S')\subseteq \cK$ is in the common link graph of
$A$ and $f(\cH_w-\{w\})$ is in the link of $x$, we can obtain a copy of $\cH$ in $\cF$ by mapping $S'$ to $A$ and $w$ to $x$, a contradiction. \qed

\medskip

Since for each $F$, $\deg_{\cB^*}(F\setminus c(F))=1$, we have $|\cB^*|=|\cD|\geq c(r,s)|\cB|$.
Hence $\cD$  satisfies both conditions in~\eqref{eq:D-requirements} with $\gamma:=c(r,s)$.
By Lemma~\ref{fine-structure}, $|\cF|\leq \binom{n}{r}-\binom{n-\sigma+1}{r}$ for sufficiently large $n$, which contradicts our earlier assumption about $|\cF|$.
This completes the proof of Theorem~\ref{sharper}.
\qed


\section{Proof of Theorem~\ref{2-tree-expansion} on $2$-tree expansions}
\label{s:2-tree-expansion}

Let $\cH$ be an $(r-2)$-reducible $r$-tree on $s$ vertices
with $\sigma(\cH)=t+1$ and $S$ a
minimum cross-cut, $\pi$  a tree-defining ordering of $\cH$ and $w$
the last vertex in $S$ that is included in $\pi$.
Let $\cH_w$ be the subgraph of $\cH$ consisting of all the edges containing $w$.
Since $\cH$ can be obtained from a $2$-uniform forest by expanding each edge into a number of
$r$-sets through expansion vertices we have that
every two edges of $\cH$ intersect in at most two vertices.

By Lemma~\ref{limb}, $\cH_w$ and $\cH'=\cH\setminus \cH_w$ are $r$-trees and
that $\exists E\in \cH_w$ and $F\in \cH'$ such that $E$ is a starting edge of $\cH_w$ and
$V(\cH_w)\cap V(\cH')= E\cap F$.
If $\cH_w$ only has one edge, then Theorem~\ref{critical-exact} applies and we
get $\ex(n,\cH)\leq \binom{n}{r}-\binom{n-\sigma+1}{r}$ and we are done.
Hence, we may assume that $\cH_w$ contains at least two edges. So $w$ has
degree at least two in $\cH$. Since $\cH$ is $(r-2)$-reducible, $E$ has at most
two vertices of degree two or higher, one of which is $w$, $w\notin V(\cH')$.
So $E\cap F$ contains at most one vertex.  Also, since $\cH_w$ 
  is a linear hypergraph, if $E\cap F$ contains
a vertex $y$ then no edge in $\cH_w$ other than $E$ contains $y$.

Let $\cF\subseteq \binom{[n]}{r}$ such that $\cH\not \subseteq \cF$.
We may assume that $|\cF|\geq \binom{n}{r}-\binom{n-\sigma+1}{r}$, since otherwise we are done.
By Lemma~\ref{fine-structure}, there exists a $(\sigma-1)$-set $A$ that satisfies~\eqref{large-link}.
Define $\cL^*,\cF_A$, and $\cS_A$ as in Lemma \ref{fine-structure}.
Define
\begin{eqnarray*}
\cB&:=&\{F\in \cF\setminus \cF_A: |F\cap V(\cL^*)|\leq 1\}\\
\cC&:=&\{F\in \cF\setminus \cF_A: |F\cap V(\cL^*)| \geq 2\}.
\end{eqnarray*}
We use Lemma \ref{fine-structure} to show that $|\cF_A\cup \cB|\leq \binom{n}{r}-\binom{n-\sigma+1}{r}$.
This holds trivially if $\cB=\emptyset$. So assume $\cB\neq \emptyset$.
Let $\cD:=\partial_{r-1}(\cB)$. Then $\cD$ is an $(r-1)$-graph on $[n]\setminus A$.
Also, $\partial_{r-2}(\cD)=\partial_{r-2}(\cB)$. Since $r\geq 4$, by the definition of
$\cB$,  $\partial_{r-2}(\cD)\cap \partial_{r-2}(\cL^*)=\emptyset$.
Since $\cH\not\subseteq \cB$, by Proposition \ref{tree-bound}, $|\cD|\geq (1/s)|\cB|$.
Hence $\cD$ satisfies both conditions in \eqref{eq:D-requirements} with $\gamma:=1/s$.
By Lemma \ref{fine-structure},
\begin{equation} \label{eq:AB}
|\cF_A\cup \cB|\leq \binom{n}{r}-\binom{n-\sigma+1}{r}.
\end{equation}

{\bf Claim~\ref{s:2-tree-expansion}.1.}\quad We have $\cH_w \not\subseteq \cC$.
\medskip

\noindent
{\it Proof of Claim~\ref{s:2-tree-expansion}.1.}   Suppose $\cC$ contains a copy $\wt{\cH}$ of $\cH_w$,
we derive a contradiction. $\cH_w$ is a linear star centered at $w$.
By earlier discussion, either $\cH_w$ is vertex disjoint from $\cH'$ or
one of its edges $E$ intersects $V(\cH')$ at one vertex $y\neq w$ and
no other edge of $\cH_w$ contains any vertex of $\cH'$. In the former case, we can
take a copy $\cK$ of $\cK^{(r-1)}_{r-1}(s)$ in $\cL^*$ and find a copy  of
$\cH'$ in $A\times \cK$ that avoids $\wt{\cH}$, which then gives us a copy
of $\cH$ in $\cF$, a contradiction.  Consider now the latter case. Since $\cH_w$
is a linear star centered at $w$, any of its edges can play the role of $E$ and any of
the vertex in $E\setminus \{w\}$ can play the role of $y$. Let $\wt{w}$ denote
the image of $w$ in $\wt{\cH}$. Let $\wt{E}$ be any edge
in $\wt{\cH}$. Since $\wt{E}\in \cC$, by definition, $|\wt{E}\cap V(\cL^*)|\geq 2$.
Let $\wt{y}\neq \wt{w}$ be a vertex in $\wt{E}\cap V(\cL^*)$.
Since $\wt{y}\in V(\cL^*)$, there exists a copy $\cK$ of $\cK^{(r-1)}_{r-1}(s)$
in $\cL^*$ that contains $\wt{y}$. Now we can easily find a copy $\cH^*$ of $\cH'$ in
$A\times \cK$ such that $y$ is mapped to $\wt{y}$ and such that $V(\cH^*)\cap
V(\wt{\cH})=\{\wt{y}\}$. This gives us a copy of $\cH$ in $\cF$, a contradiction.
\qed

By Claim~\ref{s:2-tree-expansion}.1, we have
\begin{equation} \label{eq:C}
|\cC|\leq \ex(n-\sigma+1,\cH_w).
\end{equation}

Now, since $\cF=\cF_A\cup \cB\cup \cC$, by \eqref{eq:AB} and \eqref{eq:C}, we have
$$|\cF|=|\cF\cup \cB|+|\cC|\leq \binom{n}{r}-\binom{n-\sigma+1}{r}+\ex(n-\sigma+1,
\cH_w).$$
This completes the proof of Theorem~\ref{2-tree-expansion}.
\qed


\section{Concluding remarks}

We have identified a large class
of $r$-trees $\cH$ (i.e. $2$-reducible ones) with $\ex(n,\cH)
\sim (\sigma(\cH)-1)\binom{n}{r-1}$. By contrast, Kalai's conjecture
states that for a tight $r$-trees $\cH$ on $v$ vertices $ex(n,\cH)\sim \frac{v-r}{r}\binom{n}{r-1}$. Already, the family of $1$-reducible $r$-trees lie somewhere in-between. There are $1$-reducible $r$-trees whose Tur\'an number is more
dependent on its cross-cut number and there are $1$-reducible $r$-trees whose
Tur\'an number is more dependent on its number of vertices. 
The situation with general $r$-trees is likely even more complex, providing many
intriguing questions.


\begin{thebibliography}{99}

\small

\bibitem{AHS} H. Abbott, D. Hanson, N. Sauer: Intersection theorems for systems
of sets, \emph{J. Combin. Theory Ser. A} \textbf{12} (1972), 381--389.

\bibitem{BK} N. Bushaw, N. Kettle:  Tur\'an numbers for forests of paths in hypergraphs,
\emph{SIAM J. Discree Mathematics}, \textbf{28} (2014), 711--721.

\bibitem{chvatal} V. Chv\'atal: An extremal set-intersection theorem, \emph{J. London Math. Soc.}
\textbf{9} (1974/1975), 355--359.

\bibitem{erdos-complete-partite} P. Erd\H{o}s: On extremal problems of
graphs and generalized graphs, \emph{Israel Journal of Mathematics}
\textbf{2} (1964), 183-190.


\bibitem{erdos-matching} P. Erd\H{o}s: A problem on independent $r$-tuples,
\emph{Ann. Univ. Sci. Budapest} \textbf{8} (1965), 93--95.


\bibitem{EG65} P. Erd\H{o}s, T. Gallai: On maximal paths
and circuits of graphs, {\em Acta Math. Acad. Sci. Hungar.}
\textbf{10} (1959), 337--356.


\bibitem{EKR} P.  Erd\H{o}s, C. Ko, R. Rado: Intersection theorems for systems of finite sets,
 \emph{Quart. J. Math. Oxford Ser. (2)}  \textbf{12} (1961), 313--320.


\bibitem{frankl1}
P. Frankl: On families of finite sets no two of which intersect in a singleton,
\emph{Bull. Austral. Math. Soc.}  {\bf 17} (1977),  125--134.



\bibitem{frankl-matching}
P. Frankl: Improved bounds for Erd\H os' matching conjecture,
\emph{J. Combin. Th.  Ser. A}  \textbf{120} (2013), 1068--1072.

\bibitem{FF-cluster} P. Frankl, Z. F\"uredi: A new generalization of the Erd\H{o}s-Ko-Rado theorem,
\emph{Combinatorica} \textbf{3} (1983), 341--349.


\bibitem{FF-exact} P. Frankl,  Z. F\"uredi: Exact solution of some Tur\'an-type problems,
\emph{J. Combin. Th. Ser. A} \textbf{45} (1987), 226--262.


\bibitem{FLM}
{P. Frankl, T. {\L}uczak}, { K. Mieczkowska}:
 On matchings in hypergraphs,
 \emph{Electronic J. Combin.} \textbf{19} (2012), Paper 42, 5 pp.

\bibitem{FRR}
{P. Frankl, V. R\"odl},  {A. Ruci\'nski}:
On the maximum number of edges in a triple system not containing a disjoint
 family of a given size,
\emph{Combinatorics, Probability and Computing} \textbf{21} (2012), 141--148.


\bibitem{furedi-1983} Z. F\"uredi:  On finite set-systems whose every intersection is a kernel of a star,
{\em Discrete Math.} \textbf{47} (1983), 129--132.


\bibitem{furedi-1984} Z. F\"uredi:
Hypergraphs in which all disjoint pairs have distinct unions,
  {\em Combinatorica} \textbf{4} (1984), 161--168.

\bibitem{Furedi-trees} Z. F\"uredi:
Linear trees in uniform hypergraphs, \emph{European J. Combinatorics},
\textbf{35} (2014), 264--272.


\bibitem{FJ-cycles} Z. F\"uredi, T. Jiang: Hypergraph Tur\'an numbers of linear cycles,
\emph{J. Combin. Th. Ser. A} \textbf{123} (2014), 252--270.


\bibitem{FJS} Z. F\"uredi, T. Jiang, R. Seiver:
Exact Solution of the hypergraph Tur\'an problem for $k$-uniform linear paths,
\emph{Combinatorica}, \textbf{34} (2014), 299--322.


\bibitem{Furedi-Ozkahya} Z. F\"uredi, L. \"Ozkahya: Unavoidable subhypergraphs: $a$-clusters,
\emph{J. Combin. Th. Ser. A} \textbf{118}  (2011), 2246--2256.


\bibitem{HLS}
{H. Huang, P. Loh}, {B. Sudakov}:
The size of a hypergraph and its matching number,
\emph{Combinatorics, Probability and Computing} \textbf{21} (2012), 442--450.

\bibitem{IJ} D. Irwin, T. Jiang: Tur\'an numbers of clusters, in preparation.

\bibitem{JL} T. Jiang, X. Liu: Tur\'an numbers of a class of $r$-uniform $2$-regular graphs, in preparation.

\bibitem{JPY} T. Jiang, O. Pikhurko, Z. Yilma: Set-systems without a strong simplex,
\emph{SIAM J. Discrete Math.} \textbf{24} (2010), 1038--1045.

\bibitem{Katona} G. Katona,
Intersection theorems for systems of finite sets,
\emph{Acta Math.~Acad.~Sci.~Hungar.} \textbf{15} (1964), 329--337.

\bibitem{keevash} P. Keevash: On the existence of designs, submitted. (see arXiv:1401.3665.)

\bibitem{KM} P. Keevash, D. Mubayi:  Set systems without a simplex or a cluster, \emph{Combinatorica}
\textbf{30} (2010), 175--200.

\bibitem{KMV-path-cycle} A. Kostochka, D. Mubayi, J. Verstra\"ete:  
Tur\'an problems and shadows I: Paths and Cycles, \emph{Journal of Combin. Th. Ser. A},
\textbf{129} (2015), 57--79.

\bibitem{KMV-trees} A. Kostochka, D. Mubayi, J. Versta\"ete: 
Tur\'an ptoblems and shadows II: trees, submitted. \\(See arXiv:1402.0544.)

\bibitem{L79}
{L. Lov\'asz}: \emph{Combinatorial Problems and Exercises}, Problem 13.31. 
Akad\'emiai Kiad\'o, Budapest and North Holland, Amsterdam, 1979.


\bibitem{LM}
  T. 
  {\L}uczak, K.
  Mieczkowska:
    On Erd\H{o}s' extremal problem on matchings in hypergraphs,
    \emph{J. Combin. Theory Ser. A} \textbf{124} (2014), 178–-194.

\bibitem{Mubayi-cluster} D. Mubayi:  Erd\H{o}s-Ko-Rado for three sets,
\emph{J. Combinatorial Theory Ser. A} \textbf{113} (2006), 547--550.

\bibitem{MR} D. Mubayi, R. Ramadurai: Set systems with union and intersection constraints,
\emph{J. Combinatorial Theory Ser. B} \textbf{99} (2009), 639--642.

\bibitem{MV} D. Mubayi, J. Verstra\"ete: A hypergraph extension of the bipartite
Tur\'an problem, \emph{J. Combinatorial Theory Ser A.} \textbf{106} (2004), 237-253.

\bibitem{MV2} D. Mubayi, J. Verstra\"ete: Proof of a conjecture of Erd\H{o}s on triangles in
set systems,
\emph{Combinatorica} \textbf{25} (2005), 599--614.


\bibitem{MV1} D. Mubayi, J. Verstra\"ete: Minimal paths and cycles in set systems,
\emph{European J. Combin.} \textbf{28} (2007), 1681--1693.

\bibitem{rodl} V. R\"odl: On a packing and covering problem,
\emph{European J. Combin.} \textbf{6}  (1985), 69--78.
\end{thebibliography}
\end{document}